\documentclass[11pt]{article}

\usepackage{graphics,color} 
\usepackage{epsfig,graphicx} 
\usepackage{amsmath} 
\usepackage{amssymb}  
\usepackage{psfrag}
\usepackage{yhmath}

\newcommand{\AAA}{{\mathcal A}}

\newcommand{\AC}{{\mathcal A}^{\mathsf C}}

\newcommand{\II}{{\mathbb I}}

\newcommand{\DA}{\partial \mathcal{A}}
\newcommand{\DAO}{\left [\DA\right]_{0}}

\newcommand{\DAM}{\left [\DA\right]_{\mathcal{-}}}

\newcommand{\ee}{\varepsilon}

\newcommand{\cl}{\mathsf{cl}}

\newcommand{\Int}{\mathsf{int}}

\newcommand{\sgn}{{\mathsf{sign}}}
\newcommand{\NN}{{\mathbb N}}
\newcommand{\RR}{{\mathbb R}}

\newcommand{\KK}{{\mathcal K}}
\newcommand{\UU}{{\mathcal U}}
\newcommand{\XX}{{\mathcal X}}

\newcommand{\ds}{\displaystyle}

\newcommand{\Rset}{{\mathbb R}}

\newcommand{\bBox}{\hbox{\vrule width1.3ex height1.3ex}}

\newtheorem{thm}{Theorem}[section] 
\newtheorem{pr}{Proposition}[section] 
\newtheorem{cor}{Corollary}[section] 
\newtheorem{lem}{Lemma}[section] 
\newtheorem{defn}{Definition}[section] 
\newtheorem{rem}{Remark}[section] 

\title{\textbf{On Barriers in State and Input Constrained Nonlinear  Systems}\thanks{This work was partially done while the authors were participating in the Bernoulli Program: ``Advances in the Theory of Control, Signals, and Systems, with Physical Modeling'' of the Bernoulli Center, EPFL, Switzerland in April 2009.}}

\author{Jos\'{e} A. De Don\'{a}\thanks{School of Electrical Engineering and Computer Science, Faculty of Engineering and Built Environment, Callaghan, NSW 2308, Australia.
Email: \texttt{\small Jose.Dedona@newcastle.edu.au}.
Part of this work was done while this author was visiting CAS from September 2008 to June 2009.}
\and Jean L\'{e}vine\thanks{CAS, Math\'{e}matiques et Syst\`{e}mes,
 Mines-ParisTech, 35, rue Saint-Honor\'{e}, 77300 Fontainebleau, France.
Email: \texttt{\small  jean.levine@mines-paristech.fr}.}}

\date{February 5, 2013}
\begin{document}

\maketitle

\begin{abstract}
In this paper, the problem of state and input
constrained control is addressed, with multidimensional constraints. We obtain a local description of the boundary of the admissible subset of the state space where the state and input constraints can be satisfied \emph{for all times}.
This boundary is made of two disjoint parts: the subset of the state constraint boundary on which there are trajectories pointing towards the interior of the admissible set or tangentially to it;  and a barrier, namely a semipermeable surface which is constructed via a minimum-like principle.
\end{abstract}


\section{Introduction}
In this paper, the problem of state and input constrained control is 
addressed. The objective is to describe the admissible trajectories 
of a constrained control system with as much ``freedom" as possible, 
without introducing additional exogenous elements. State constraints 
have mostly been addressed theoretically in the framework of optimal 
control (see e.g. Chapter~VI of \cite{PBGM}; see also 
\cite{Clarke,Vinter,Trelat} for extensions and more modern 
presentations). The approach proposed here is complementary, since 
we focus attention on the characterisation and computation of the 
boundary of the admissible region, namely the subset of the state space where the state and input constraints can be satisfied \emph{for all times}. A distinctive feature of the problem 
considered in this paper, compared to optimal control, is that we 
are in a \emph{qualitative} situation, since there is no \emph{a priori} 
notion of a quantitative objective function to be optimised, nor of a target to reach.

Related mathematical ideas may be found in the characterisation of 
other aspects of control systems, namely attainable regions 
\cite{Fein}, reachable sets \cite{Son}, invariant sets \cite{Blan-book,WB} and problems related 
to optimal control \cite{PBGM,Gam} and differential games \cite{Isaacs,MBT}.
Some of the results presented in this paper may be also interpreted in 
terms of viability kernels, a major concept of viability theory \cite{Aubin}. 
In particular, see the work of Quincampoix \cite{Quincampoix_siam} on 
differential inclusions and target problems. Numerical studies of 
constrained trajectories, extending the work of Quincampoix 
\cite{Quincampoix_siam}, may be found, e.g., in~\cite{Cr}. 

In contrast with the latter works, by dealing with ordinary differential 
equations we develop a direct mixed topological and geometric approach leading to 
computable conditions. We prove a \emph{minimum-like principle} satisfied 
by the barrier, the semipermeable subset of the boundary of the admissible 
set.

The remainder of the paper is organised as follows. In Section~\ref{sec:MotEx}
we introduce an example that provides physical intuition and motivates the 
problem and the main ideas. In Section~\ref{sec:ConsDynCon} we more precisely 
define the problem we want to address, introduce the notation and state some 
assumptions on the dynamical system and the constraints. 
In Section~\ref{sec:AdmSetTopol} we investigate the topological properties of 
the admissible set. In Sections~\ref{sec:BoundAdmSet} and~\ref{sec:UltTangCond} 
we investigate the properties of the boundary of the admissible region and derive the
conditions under which that boundary intersects the state constraint boundary, 
which we call ultimate tangentiality conditions. In Section~\ref{sec:BarrEqn}
we derive the minimum-like principle condition satisfied by the barrier and discuss its
relationship with notions from optimal control theory. Section~\ref{sec:Examples}
presents a number of examples of linear and nonlinear constrained systems 
with a complete characterisation of the admissible region. Section~\ref{sec:Conclus}
presents the conclusions of the paper, and some technical auxiliary proofs
concerning compactness of solutions of differential equations, followed by some of their variational properties and a recall of the maximum principle, are presented in Appendices~\ref{Append-A} and \ref{Append-B} at the end of the paper.

\section{A motivating example: Single-particle constrained motion}\label{sec:MotEx}

We start our development by considering a simple example,
consisting of a particle of mass $m$ moving in space with velocity $V(t)$. We
also consider that the motion of the particle is constrained to remain
in some region of the space and that we can apply a force $F(t)$ to the
particle in any direction, but with the constraint
$|F(t)|\leq 1$. This scenario is represented in Figure~\ref{fig:Figure_1}.
\begin{figure}[thpb]
\begin{center}
\includegraphics[width=.55\columnwidth]{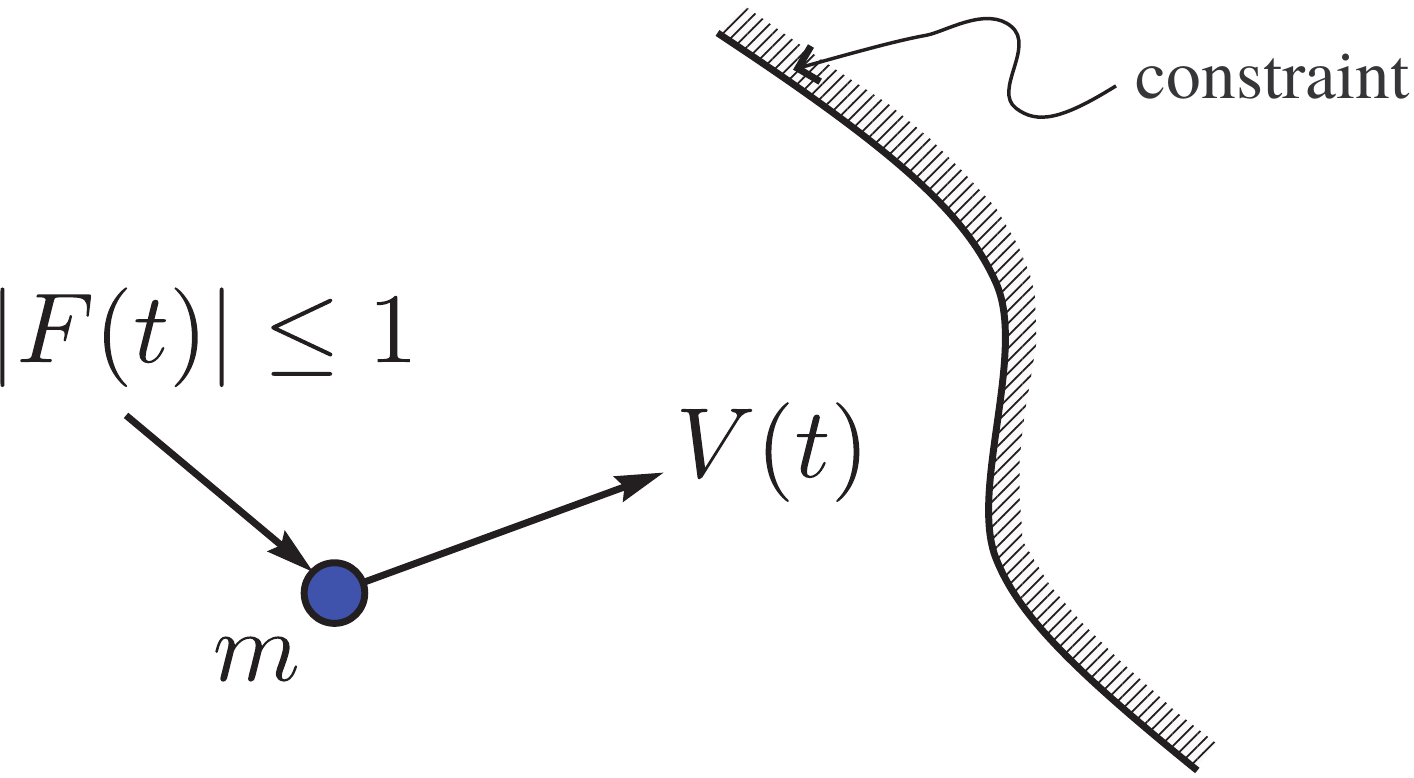}
\caption{Constrained motion of a particle.}
\label{fig:Figure_1}
\end{center}
\end{figure}
An important feature of this problem, that we want to emphasise here, is
that we consider that the constraints (both, in force and position) have
to be satisfied for all times, and not just over a finite time window.

The first question is: Given our knowledge of the position $X(t)$ of the
particle and its velocity $V(t)$ at a given time instant $t$, what action should
we take/start taking, given the additional knowledge of the constraint?
Are we allowed to move into any region of the space, provided we remain
(in the figure above) to the left of the constraint surface?

Consider the limit case  represented in Figure~\ref{fig:Figure_2}.
\begin{figure}[thpb]
\begin{center}
\includegraphics[width=.35\columnwidth]{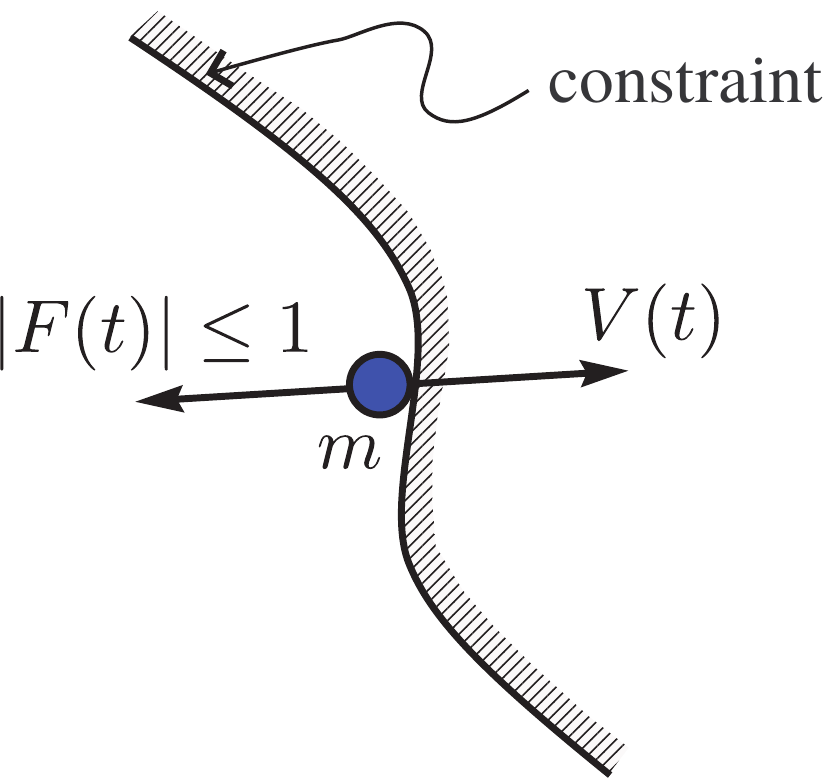}
\caption{A ``limit situation".}
\label{fig:Figure_2}
\end{center}
\end{figure}
It is evident that such a situation can not be allowed, since from
Newton's second law we know it would not be possible
to prevent the particle from entering the forbidden region. (Only a
force infinite in magnitude would be able to instantaneously change the
direction of the particle, a situation that our constrained input does
not allow.)

An intuitively simple solution would be to think of a ``buffer region'' of thickness
$\varepsilon$ that would allow our particle enough time to change direction steered
by the constrained force. As represented in Figure~\ref{fig:Figure_3}, it is to
be expected that such a buffer region will depend on both, the position
$X$ of the particle and its velocity $V$ at that specific position.
\begin{figure}[thpb]
\begin{center}
\includegraphics[width=.40\columnwidth]{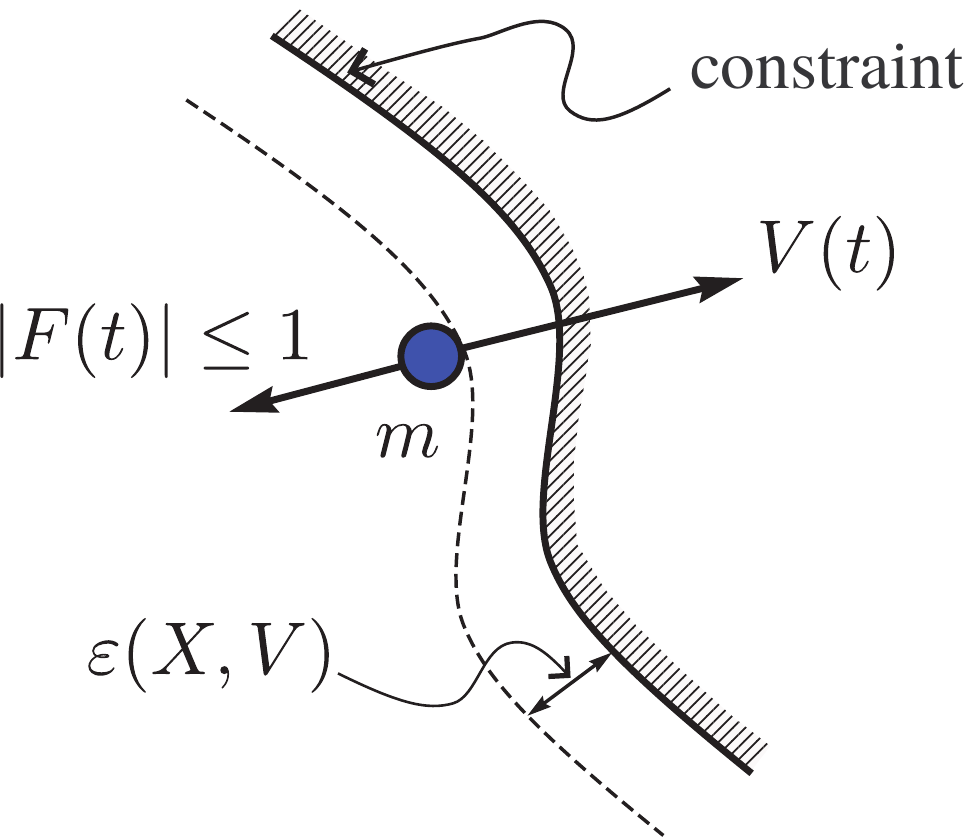}
\caption{A ``buffer" region.}
\label{fig:Figure_3}
\end{center}
\end{figure}
Many ways of constructing such a ``buffer region'' could be devised, for
example one could construct an ``artificial potential'' such that the
particle, once entered into the region, is repelled back or, at least, is not allowed
to overpass the constraint (see e.g. \cite{BMPR}). By means of influencing $F$ we could, for instance,
introduce ``artificial damping'' so as to
dissipate the energy of the particle and prevent it from trespassing
the constraint. However, these---potentially useful in practice---ingenious ways
of tackling the problem would seem to defeat our initial
purpose of ``parameterising with as much freedom as possible" all the possible constrained ``natural"
trajectories of the system (without adding artificial exogenous elements).

Still, motivated by the previous considerations, we will analyse one last
fact related to the simple single-particle example discussed above.
It is obvious that, whatever the particle's trajectories do in
space, the moment they enter in contact with the constraint surface
they must do so with a velocity that is tangential to the constraints
(note that these limit trajectories cannot be perturbed without the risk of violating the constraints). This situation is represented in Figure~\ref{fig:Figure_4}.

\begin{figure}[thpb]
\begin{center}
\includegraphics[width=.55\columnwidth]{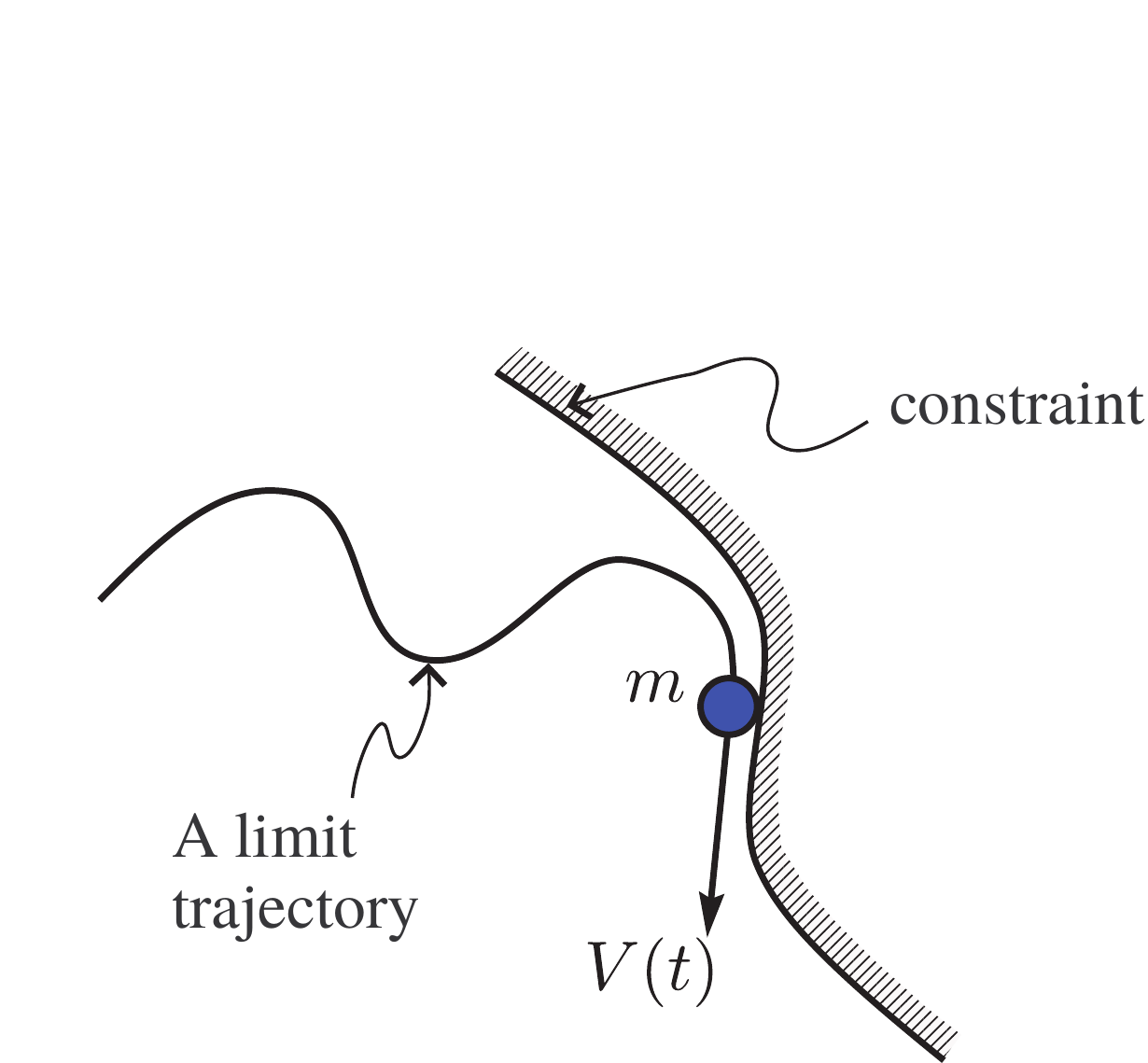}
\caption{A limit trajectory arriving tangentially to the constraint.}
\label{fig:Figure_4}
\end{center}
\end{figure}
In general terms, in this paper we prove that such a tangency property is necessary for a trajectory to remain on the boundary of the \emph{admissible set}, namely the set of all initial states from which there exists a trajectory satisfying the constraints for all times. We further prove a general minimum-like principle to construct this boundary, which arrives tangentially to the constraint.

\section{Constrained dynamical control systems}\label{sec:ConsDynCon}

We consider the following constrained nonlinear system:
\begin{align}
\label{eq:state_space}
  & \dot{x} =  f(x,u), \\
  \label{eq:initial_condition}
  & x(t_0) = x_0, \\
  \label{eq:input_constraint}
 & u  \in \UU, \\
  \label{eq:state_const}
& g_{i}\big(x(t)\big)  \leq   0 \quad \forall t \in [t_0, \infty), \quad \forall i \in \{1, \ldots, p \}
\end{align}
where $x(t)\in \Rset^{n}$, $\Rset^{n}$ being endowed with the usual topology of the Euclidean norm.
We denote by $U_{1}\triangleq \{ v\in \Rset^{m} : \Vert v \Vert \leq1\}$ the closed unit ball\footnote{We keep the same notation $\Vert\cdot\Vert$ for the Euclidean norm of $\Rset^{p}$ for every $p\geq 1$.} of $\Rset^{m}$.
The input function $u$ is assumed to belong to the set $\UU$ of Lebesgue measurable functions from $[t_0, \infty)$ to  $U_{1}$, i.e. $u$ is a measurable function such that $\Vert u(t)\Vert \leq 1$ for all $t\in [t_0, \infty)$. 

We need to recall from \cite{PBGM}  the definition of \emph{Lebesgue} or \emph{regular} point for a given control $u\in \UU$: the time $t\in [t_0,\infty)$ is called a \emph{Lebesgue point} for $u$ if $u$ is continuous at $t$ in the sense that there exists a bounded (possibly empty) subset $I_{0} \subset [t_0,\infty)$, of zero Lebesgue measure, with $t\not \in I_{0}$, such that $u(t)= \lim_{s\rightarrow t, s\not \in I_{0}} u(s)$.
Since $u$ is Lebesgue-measurable, by Lusin's theorem, the Lebesgue measure of the complement, in $[t_0,T]$, for all finite $T$, of the set of Lebesgue points is equal to 0.

Note that if $u_{1}\in \UU$ and $u_{2}\in \UU$, and if $\tau\geq t_{0}$ is given, the concatenated input $v$, defined by $v(t)= \left\{ \begin{array}{ll} u_{1}(t)&\mbox{\textrm if~} t\in [t_{0}, \tau[\\u_{2}(t)&\mbox{\textrm if~} t \geq \tau\end{array}\right.$ satisfies $v\in \UU$. The concatenation operator relative to $\tau$ is denoted by $\Join_{\tau}$, i.e. $v=u_{1}\Join_{\tau} u_{2}$. 

To easily handle the multidimensional constraints \eqref{eq:state_const}, let us introduce the following notations. The \emph{constraint set} is defined as:
\begin{equation}\label{Gdef}
  G \triangleq \{x \in \mathbb{R}^n:  g_{i}(x)\leq 0, i=1,\ldots,p \}.
\end{equation}
\begin{itemize}
\item Assuming that 
$G$, endowed with its relative topology in $\RR^{n}$,  has nonempty interior, denoting $g(x)= (g_{1}(x), \ldots, g_{p}(x))$, we denote by $g(x)\circeq 0$ the fact that $x\in \RR^{n}$ satisfies $g_{i}(x)=0$ for at least one $i\in \{ 1, \ldots, p \}$ and $g_{i}(x)\leq 0$ for all $i\in \{ 1, \ldots, p \}$. The set of all indices $i\in \{ 1, \ldots, p \}$ such that $g_{i}(x)=0$ is denoted by $\II(x)$. 
\item By $g(x) \prec 0$ (resp. $g(x) \preceq 0$) we mean that $g_{i}(x) < 0$ (resp. $g_{i}(x) \leq 0$) for all $i\in \{ 1, \ldots, p \}$.
\end{itemize}

We also define the sets 
\begin{equation}\label{G0def}
  G_0 \triangleq \{x \in \mathbb{R}^n:  g(x)\circeq 0 \}, \qquad
  G_- \triangleq \{x \in \mathbb{R}^n: g(x) \prec 0\}.
\end{equation}
Clearly, the constraint set is given by $G = G_0 \cup G_-$. We call the points of $G_0$ boundary points.

\bigskip

We further assume:
\begin{description}
\item[(A1)] $f$ is an at least $C^{2}$ vector field of $\RR^{n}$ for every $u$ in an open subset of $\Rset^{m}$ containing $U_{1}$, whose dependence with respect to $u$ is also at least $C^{2}$.
\item[(A2)] There exists a constant $0 < C < +\infty$ such that one of the following inequalities holds true:
\begin{itemize}
\item[(i)] $\sup_{u\in U_{1}}\Vert f(x,u) \Vert \leq C(1+ \Vert x \Vert )$ for all $x$;
\item[(ii)] $\sup_{u\in U_{1}}\vert x^{T}f(x,u) \vert \leq C(1+ \Vert x \Vert^{2} )$ for all $x$.
\end{itemize}
\item[(A3)] The set $f(x,U_1)$, called the \emph{vectogram} in \cite{Isaacs}, is convex for all $x\in \RR^{n}$.
\item[(A4)] $g$ is an at least $C^{2}$ function from $\RR^{n}$ to $\RR^{p}$; and for each $i=1,\ldots, p$, the set of points given by $g_{i}(x)=0$ defines an $n-1$ dimensional manifold (also called, loosely speaking, a \emph{face} of the set $G$).
\end{description}

\bigskip

In (\ref{eq:state_const}),
$x(t)$---sometimes denoted
$x^u(t)$, or $x^{(u,x_0)}(t)$, or even $x^{(u,x_0,t_0)}(t)$, when the distinction is required---denotes the solution of the differential equation~\eqref{eq:state_space} with input  $u\in \UU$  and initial condition~\eqref{eq:initial_condition}.

Going back to the concatenation operator, it is readily verified that the integral curve $x^{(v,x_0,t_0)}$ generated by $v = u_{1}\Join_{\tau} u_{2}$ from $x_{0}$ at initial time $t_{0}$ satisfies 
$$x^{(v,x_0,t_0)} = x^{(u_{1},x_{0},t_{0})}\Join_{\tau} x^{(u_{2},x^{(u_{1},x_{0},t_{0})}(\tau), \tau)}$$ 
i.e. coincides, on $[t_{0},\tau]$, with the arc of integral curve generated by $u_{1}$ on this interval, and, on $[\tau,\infty)$, with the integral curve generated by $u_{2}$ starting at time $\tau$ from the end point $x^{(u_{1},x_{0},t_{0})}(\tau)$ of the previous one.

Several results, concerning compactness of the solutions of
system~\eqref{eq:state_space} and some of their variational properties, that are key in various proofs of the remaining sections, and which are available in the literature in a rather scattered way and embedded in slightly different contexts (see e.g. \cite{Cesari,Lee_Markus,Trelat}), are presented for completeness in the Appendices~\ref{Append-A} and \ref{Append-B} at the end of the paper.

Without loss of generality, since system (\ref{eq:state_space}) is time-invariant, $t_0$ may be replaced by 0. For simplicity's sake, we adopt this choice in the sequel and, when clear from the context, ``$\forall t$" or ``for almost all $t$'' will mean
``$\forall t \in [0, \infty)$" or ``for almost all $t\in [0, \infty)$''.

\section{The admissible set. Topological properties}\label{sec:AdmSetTopol} 

In this section we define the admissible set for the constrained nonlinear system~\eqref{eq:state_space}--\eqref{eq:state_const} and study its topological properties.

\begin{defn}[Admissible States]
\label{def:admiss_states}
We will say that a state-space point $\bar{x}$ is \emph{admissible} if there exists, at least, one input function $v\in \UU$, such
that~\eqref{eq:state_space}--\eqref{eq:state_const} are
satisfied for $x_0=\bar{x}$ and $u=v$. Note that the Markovian property of the system implies that any point of the integral curve,
$x^{(v,\bar{x})}(t_1)$, $t_1 \in [0, \infty)$, is also an admissible point. The set of admissible states is expressed mathematically as:
\begin{equation}\label{eq:Admiss_states}
 \AAA \triangleq \{\bar{x} \in G: \exists u\in \UU,~ g\big(x^{(u,\bar{x})}(t)\big) \preceq  0, \forall t\}.
\end{equation}
\end{defn}

Its complement in $G$, namely $ \AC \triangleq G\setminus\AAA$, is thus given by:
\begin{equation}\label{eq:Compl_Admiss_states}
 \AC \triangleq \{\bar{x} \in G: \forall u\in \UU,~\exists \bar{t} < +\infty, \exists i\in \{1,\ldots,p\}
 ~ \mathrm{s.t.} ~ g_{i}\big(x^{(u,\bar{x})}(\bar{t})\big) >  0\}.
\end{equation}
From now on, all set topologies will be defined relative to $G$.

We discard the trivial cases  $\AAA = \emptyset$ and $\AC = \emptyset$. Therefore, in the sequel, we assume that both  $\AAA$ and $\AC$ contain at least one element.

In addition to the infinite horizon admissible set $\AAA$, we also consider the family of sets $\AAA_{T}$, called finite horizon admissible sets, defined for all finite $0\leq T<+\infty$ by
$$\AAA_T  \triangleq \{\bar{x} \in G: \exists u\in \UU,~ g\big(x^{(u,\bar{x})}(t)\big) \preceq  0, \forall t \leq T\}.$$

Accordingly, its complement $\AC_{T}$ in $G$ is given by:
$$\AC_{T} \triangleq \{\bar{x} \in G: \forall u\in \UU,~\exists \bar{t} \leq T, \exists i\in \{1,\ldots,p\}
 ~ \mathrm{s.t.} ~ g_{i}\big(x^{(u,\bar{x})}(\bar{t})\big) >  0\}.$$

\begin{rem}
We stress that the study of the admissible set $\AAA$ in Definition~\ref{def:admiss_states} cannot be reduced to the one of attainable sets (also called reachable sets). 
Recall (see e.g. \cite{Isidori, Lee_Markus, NvdS, Son}) that the attainable, or reachable, set at time $t$ from $\bar{x}$, denoted by $X_{t}(\bar{x})$, is the subset of $\RR^{n}$ defined by
\begin{equation}\label{attain-set}
X_{t}(\bar{x})\triangleq \{ x^{(u,\bar{x})}(t): u\in \UU \}.
\end{equation}
Indeed, one can prove that\footnote{Though nowhere used in this paper, identities (\ref{adm-att}) are provided for the sake of completeness. Their proof, based on the existence of an absolutely continuous section $\{ (t,x(t)) : x(t) \in X_{t}(\bar{x}) \cap G \;\; \forall t \}$, is left to the reader.}
\begin{equation}\label{adm-att}
\AAA_T = \{ \bar{x} \in \RR^{n}: X_{t}(\bar{x}) \cap G \neq \emptyset , \; \forall t\in [0,T] \}, \;\; \AAA = \{ \bar{x} \in \RR^{n}: X_{t}(\bar{x}) \cap G \neq \emptyset, \;  \forall t \geq 0 \}.
\end{equation}

The set $\AAA$ can also be related to the notion of \emph{maximal controlled positively invariant set}, see e.g.~\cite{Blan-book}
and the references therein. We point out that, as opposed to the latter references, our definition is concerned with general measurable controls in the classical context of ordinary differential equations, and does not depend on the existence of a synthesizable feedback law.
\end{rem}

\begin{pr}
\label{closedness-prop}
Assume that (A1)--(A4) are valid. The set of finite horizon admissible states, $\AAA_{T}$, is closed for all finite $T$.
\end{pr}
\begin{proof}
The result is an immediate consequence of the compactness result proven in Lemma~\ref{compact-lem} of  Appendix~\ref{Append-A}. We sketch it for the sake of completeness.

Consider a sequence of initial states $\{ x_{k} \}_{k\in \NN}$ in $\AAA_T$ converging to $\bar{x}$ as $k$ tends to infinity. By definition of $\AAA_T$, for every $k\in \NN$, there exists $u_{k}\in \UU$ such that the corresponding integral curve $x^{(u_{k},x_{k})}$ satisfies $g(x^{(u_{k},x_{k})}(t)) \preceq 0$ for all $t\in [0,T]$. According to Lemma~\ref{compact-lem}, there exists a uniformly converging subsequence, still denoted by $x^{(u_{k},x_{k})}$, to the absolutely continuous integral curve $x^{(\bar{u},\bar{x})}$ for some $\bar{u}\in \UU$. By the continuity of $g$, we immediately get that
$g(x^{(\bar{u},\bar{x}))}(t)) \preceq 0$ for all $t\in [0,T]$, hence $\bar{x}\in \AAA_T$, and the
proposition is proven.
\end{proof}

\begin{cor}\label{close-cor}
Under the assumptions of Proposition~\ref{closedness-prop}, the set $\AAA$ is closed.
\end{cor}
\begin{proof}
For all $0\leq T_{1} \leq T_{2} < \infty$, we have:
\begin{equation}\label{incl}
\AAA=\AAA_{\infty}\subset \AAA_{T_{2}}\subset \AAA_{T_{1}}\subset \AAA_0 = G.
\end{equation}
Therefore, $\AAA=\bigcap_{T\geq 0} \AAA_T$ and the result is an
immediate consequence of the fact that the intersection of a family of closed sets is closed.
\end{proof}
\begin{rem}
Condition (A2) on $f$ is introduced here in order to guarantee boundedness and uniform convergence of a sequence of integral curves, which are required in the proof of the above proposition (see the details in  Appendix~\ref{Append-A}). However, as the examples shown later will
illustrate, this (sufficient) condition is generally not needed since many systems that
do not satisfy it will still have bounded trajectories under admissible controls. Any other condition on $f$ ensuring uniform boundedness (see e.g. \cite{Filippov_siam}) will give similar compactness results.
\end{rem} 

\section{Boundary of the admissible set}\label{sec:BoundAdmSet}

We conclude from Proposition~\ref{closedness-prop} and Corollary~\ref{close-cor} that the sets $\AAA_{T}$, for all $T\geq 0$, and $\AAA$ contain their respective boundaries, denoted by $\DA_{T}$ and $\DA$ respectively.
From now  on, we will focus on the properties and
characterisation of these boundaries, which are arguably the most important objects to describe the effect of the constraints~(\ref{eq:input_constraint})--(\ref{eq:state_const}) on the time evolution of the dynamical system~(\ref{eq:state_space}).

To start with, let us adapt the definition of semipermeability (a notion introduced several decades ago in differential games by R.\ Isaacs \cite{Isaacs}) to our context. 
\begin{defn}[Semipermeability]\label{semiperm-def}
The boundary $\partial A$ (relative to $B$) of a closed set $A\subset B \subset \RR^{n}$ is called semipermeable relative to $B$ and system (\ref{eq:state_space}), or just semipermeable if non-ambiguous, if, and only if, no integral curve of the system starting in $A^{\mathsf C}\cap B$ can penetrate $A$ before leaving $B$, i.e. if, and only if, for all $x_{0} \in A^{\mathsf C}\cap B$, for all $u\in \UU$ and all $t$ such that the connected arc of integral curve $\wideparen{x_{0}, x^{(u,x_{0})}(t)}\subset B$ we have that $x^{(u,x_{0})}(t) \in A^{\mathsf C}$.
\end{defn}

\begin{pr}\label{semiperm:prop}
For all $T\geq 0$, the sets $\DA_{T} \cap G_{-}$ and $\DA \cap G_{-}$ are semipermeable relative to $G$.
\end{pr}
\begin{proof}
We directly prove this result for $T=+\infty$. The proof for finite $T$ follows exactly the same lines.

Assume that $\bar{x} \in \AC \cap G_{-}$. For every $u\in \UU$, we denote by $\bar{t}(u)$ the first exit time from $G$, i.e.
$$\bar{t}(u) = \inf \{ t\in [0,\infty ) : \; \exists i\in \{1,\ldots,p\} ~s.t.~ g_{i}(x^{(u,\bar{x})}(t)) > 0\}.$$ 
Such a lower bound exists for every $u\in \UU$ by definition of $\AC$. If $\bar{t}(u)=0$ for every $u\in \UU$, then, clearly, $\bar{x} \not\in G_{-}$. Therefore, there exists $\tilde{u}\in \UU$ such that $\bar{t}(\tilde{u}) > 0$.
Let us assume by contradiction that for such $\tilde{u}$ there exists  $\tilde{t} \in \ ]0, \bar{t}(\tilde{u})[$ such that $\tilde{\xi}\triangleq x^{(\tilde{u},\bar{x})}(\tilde{t}) \in \AAA$ and we set $\tilde{\tilde{u}} =  \tilde{u} \Join_{\tilde{t}} \bar{u}$ where $\bar{u}$ is such that $g(x^{(\bar{u},\tilde{\xi})}(t)) \preceq 0$ for all $t \geq \tilde{t}$, which exists by definition of $\AAA$.
The corresponding integral curve $x^{(\tilde{\tilde{u}}, \bar{x})}$ is therefore equal to $x^{(\tilde{u},\bar{x})}$ for $t \leq \tilde{t}$ and to $x^{(\bar{u},\tilde{\xi})}$ for $t \geq \tilde{t}$. Moreover, since $g(x^{(\tilde{u},\bar{x})}(t)) \preceq 0$ for $t \leq \tilde{t}$ and $g(x^{(\bar{u},\tilde{\xi})}(t)) \preceq 0$ for $t \geq \tilde{t}$, we conclude that
$g(x^{(\tilde{\tilde{u}}, \bar{x})}(t)) \preceq 0$ for all $t \in [0,\infty)$. Therefore, we have proven that $\bar{x} \in \AAA$, which contradicts our assumption that $\bar{x} \in \AC \cap G_{-}$. It results that no such $\tilde{t}$ can exist, and therefore that no trajectory starting in $\AC \cap G_{-}$ can penetrate $\AAA$ before leaving $G$, thus proving that $\DA \cap G_{-}$ is semipermeable relative to $G$. 
\end{proof}

We now prove the following result.

\begin{pr} \label{min-sup-prop}
Assume that (A1)--(A4) hold.  We have the following equivalences:
\begin{itemize}
\item[(i)] $\bar{x} \in \AAA$ is equivalent to
\begin{equation}\label{AAAopt}
\min_{u\in \UU} \sup_{t\in [0,\infty)} \max_{i=1,\ldots,p} g_{i}(x^{(u,\bar{x})}(t)) \leq 0
\end{equation}
\item[(ii)]Ê$\bar{x} \in \AC$ is equivalent to
\begin{equation}\label{ACopt}
\min_{u\in \UU} \sup_{t\in [0,\infty)} \max_{i=1,\ldots,p} g_{i}(x^{(u,\bar{x})}(t)) >  0
\end{equation}
\item[(iii)] $\bar{x} \in \DA$ is equivalent to
\begin{equation}\label{DAopt}
\min_{u\in \UU} \sup_{t\in [0,\infty)} \max_{i=1,\ldots,p} g_{i}(x^{(u,\bar{x})}(t)) = 0.
\end{equation}
\end{itemize}
\end{pr}
\begin{proof}
We first prove (i). If $\bar{x} \in \AAA$, by definition, there exists $u\in \UU$ such that $g(x^{(u,\bar{x})}(t))\preceq 0$ for all $t\geq 0$, and thus such that $\sup_{t\in [0,\infty)}  \max_{i=1,\ldots,p} g_{i}(x^{(u,\bar{x})}(t))\leq 0$. We immediately get
\begin{equation}\label{infless}
\inf_{u\in \UU} \sup_{t\in [0,\infty)}  \max_{i=1,\ldots,p} g_{i}(x^{(u,\bar{x})}(t)) \leq 0.
\end{equation}
 Let us prove next that the infimum with respect to $u$ is achieved by some $\bar{u}\in \UU$ in order to get (\ref{AAAopt}). To this aim, let us consider a minimising sequence $u_{k} \in \UU$, $k\in \NN$, i.e. such that
\begin{equation}\label{minseq}
\lim_{k\rightarrow \infty}  \sup_{t\in [0,\infty)}  \max_{i=1,\ldots,p} g_{i}(x^{(u_{k},\bar{x})}(t)) = \inf_{u\in \UU} \sup_{t\in [0,\infty)}  \max_{i=1,\ldots,p} g_{i}(x^{(u,\bar{x})}(t)).
\end{equation}
According to Lemma~\ref{compact-lem} in  Appendix~\ref{Append-A}, with $x_{k}=\bar{x}$ for every $k\in \NN$, one can extract a uniformly convergent subsequence on every compact interval $[0,T]$ with $T\geq 0$, still denoted by $x^{(u_{k},\bar{x})}$, whose limit is $x^{(\bar{u},\bar{x})}$ for some $\bar{u}\in \UU$. The continuity of $g$ implies that the sequence $g(x^{(u_{k},\bar{x})})$ uniformly converges to the function $g(x^{(\bar{u},\bar{x})})$ on every compact interval $[0,T]$ with $T\geq 0$. Therefore, for every $T\geq 0$ and every $\ee >0$, there exists $k_{0}(T,\ee) \in \NN$ such that, for every $k\geq k_{0}(T,\ee)$ and every $i \in \{ 1, \ldots, p \}$, we have $g_{i}(x^{(\bar{u},\bar{x})}(t)) \leq g_{i}(x^{(u_{k},\bar{x})}(t)) +\ee$, $\forall t \in [0,T]$; which yields, taking the supremum w.r.t.\ $t\in [0,\infty)$ and $i \in \{ 1, \ldots, p \}$ on the right hand side 
\begin{equation*}
g_{i}(x^{(\bar{u},\bar{x})}(t)) \leq \sup_{t\in [0,\infty)} \max_{i=1,\ldots, p}g_{i}(x^{(u_{k},\bar{x})}(t)) + \ee 
\qquad \forall t \in [0,T]. 
\end{equation*}
On the other hand, by the definition of the limit in (\ref{minseq}), for every $\ee >0$ 
there exists $k_{1}(\ee) \in \NN$ such that for all $k\geq k_{1}(\ee)$, we have
$$\sup_{t\in [0,\infty)}  \max_{i=1,\ldots, p}g_{i}(x^{(u_{k},\bar{x})}(t)) \leq  \inf_{u\in \UU} \sup_{t\in [0,\infty)}  \max_{i=1,\ldots, p}g_{i}(x^{(u,\bar{x})}(t)) + \ee.$$ 
Thus, for all $k\geq \max(k_{0}(T,\ee),k_{1}(\ee))$, we get
$$
g_{i}(x^{(\bar{u},\bar{x})}(t)) \leq  \inf_{u\in \UU} \sup_{t\in [0,\infty)}  \max_{i=1,\ldots, p}g_{i}(x^{(u,\bar{x})}(t)) + 2\ee \quad \forall t\in [0,T], \quad \forall i=1,\ldots,p.
$$
However, since the latter inequality is valid for any
$T \geq 0$ and it does not depend on $k$ anymore, and since its right-hand side is independent of $i$,  $t$ and $T$, we have that the inequality holds if we take the supremum of the left-hand side with respect to $t\in [0,\infty)$ and $i \in \{ 1, \ldots, p \}$. Using, in addition, the definition of the infimum w.r.t.\ $u$, we thus obtain that
for every $\ee > 0$
$$
\begin{aligned}
\sup_{t\in [0,\infty)} \max_{i=1,\ldots,p} g_{i}(x^{(\bar{u},\bar{x})}(t)) 
&\leq  \inf_{u\in \UU} \sup_{t\in [0,\infty)}  \max_{i=1,\ldots, p}g_{i}(x^{(u,\bar{x})}(t)) + 2\ee 
\\&\leq  \sup_{t\in [0,\infty)}  \max_{i=1,\ldots, p}g_{i}(x^{(\bar{u},\bar{x})}(t))+ 2\ee,
\end{aligned}
$$
or, using also~(\ref{infless}), that
$$ \sup_{t\in [0,\infty)}  \max_{i=1,\ldots, p}g_{i}(x^{(\bar{u},\bar{x})}(t)) = \inf_{u\in \UU} \sup_{t\in [0,\infty)}  \max_{i=1,\ldots, p}g_{i}(x^{(u,\bar{x})}(t)) \leq 0,$$
which proves (\ref{AAAopt}).

Conversely, if (\ref{AAAopt}) holds, there exists an input $u\in \UU$ such that
$$\sup_{t\in [0,\infty)}  \max_{i=1,\ldots, p}g_{i}(x^{(u,\bar{x})}(t)) \leq 0,$$
which in turn implies that
$g(x^{(u,\bar{x})}(t))\preceq 0$ for all $t\geq 0$, or, in other words, $\bar{x} \in \AAA$, which achieves the proof of (i).

To prove (ii), we now assume that $\bar{x} \in \AC$ and prove (\ref{ACopt}). By definition of $\AC$, for all $u\in \UU$, we have $\sup_{t\in [0,\infty)}  \max_{i=1,\ldots, p}g_{i}(x^{(u,\bar{x})}(t)) > 0$ and thus
$$\inf_{u\in \UU}\sup_{t\in [0,\infty)}  \max_{i=1,\ldots, p}g_{i}(x^{(u,\bar{x})}(t)) \geq 0.$$
The same minimising sequence argument as in the proof of (i) shows that the minimum over $u\in \UU$ is achieved by some $\bar{u}\in \UU$ and that
$$\min_{u\in \UU}\sup_{t\in [0,\infty)}  \max_{i=1,\ldots, p}g_{i}(x^{(u,\bar{x})}(t)) \geq 0.$$
But the inequality has to be strict since, if $\ds \min_{u\in \UU}\sup_{t\in [0,\infty)}  \max_{i=1,\ldots, p}g_{i}(x^{(u,\bar{x})}(t)) = 0$, it would imply, according to (i), that $\bar{x}\in \AAA$ which contradicts the assumption. Therefore, we have proven  (\ref{ACopt}).

Conversely, if  (\ref{ACopt}) holds, it is immediately seen that $\bar{x}$ is such that \sloppy $\ds \sup_{t\in [0,\infty)}  \max_{i=1,\ldots, p}g_{i}(x^{(u,\bar{x})}(t)) > 0$ for all $u\in \UU$. Since the mapping $\ds t \mapsto   \max_{i=1,\ldots, p}g_{i}(x^{(u,\bar{x})}(t))$ is continuous, the inverse image $\ds \{ t \in [0, \infty) :  \max_{i=1,\ldots, p}g_{i}(x^{(u,\bar{x})}(t)) > 0\}$ is a nonempty open subset of  $[0, \infty)$. Thus, for all $u\in \UU$, there exists $\bar{t}(u) < +\infty$ such that $\ds \max_{i=1,\ldots, p}g_{i}(x^{(u,\bar{x})}(\bar{t}(u))) > 0$, and hence $\bar{x}\in \AC$, which proves (ii).

To prove (iii), since $\AAA$ is closed, $\bar{x} \in \DA$ is equivalent to $\bar{x}\in \AAA$ and $\bar{x}\in \cl(\AC)$, the closure of $\AC$, which, by (i) and (ii), is equivalent to (\ref{AAAopt}) and (\ref{ACopt}) (the latter with a ``$\geq$'' symbol as a consequence of $\bar{x}\in \cl(\AC)$), which in turn is equivalent to  
(\ref{DAopt}).
\end{proof}

\begin{rem}
The same formulas hold true for $\AAA_T$, $\AC_T$ and $\DA_T$ if one replaces the infinite time interval $[0,\infty)$ by $[0,T]$.
\end{rem}

We define the sets
\begin{equation}\label{da0-dam-eq}
\DAO= \DA \cap G_0, \quad  \DAM= \DA \cap G_-.
\end{equation}
They indeed satisfy $\DA = \DAO \cup \DAM$.

We denote by $L_{f}h(x,u)\triangleq Dh(x)f(x,u)$ the Lie derivative of a smooth function $h:\RR^{n} \rightarrow \RR$ along the vector field $f(\cdot,u)$ at the point $x$.

The boundary subset $\DAO$ is easily characterised.
\begin{pr}\label{pr:usable}
The boundary subset $\DAO$ is contained in the set of points $z\in G_{0}$ such that 
\begin{equation}\label{us-boundary:eq}
\min_{u\in U_{1}} \max_{i \in \II(z)} L_{f} g_{i}(z,u) \leq 0
\end{equation}
with the notations introduced in Section~\ref{sec:ConsDynCon}. 
\end{pr}
\begin{proof}
Since $\DAO \subset \AAA$, for all $z \in \DAO$, there must exist $u\in \UU$ and an integral curve $x^{(u,z)}$ which remains in $\AAA$ for all times. Therefore, at $t=0$, the vector $f(z,u(0+))$, where $u(0+)$ is the right limit of $u(t)$ when $t\searrow 0$, cannot point outwards of $G_{-}$, i.e. $L_{f}g_{i}(z,u(0+))\leq 0$ for all $i\in \II(z)$; otherwise, if there exists $i \in \II(z)$ such that $L_{f}g_{i}(z,u(0+)) > 0$, then the integral curve $x^{(u,z)}$ instantaneously leaves $G$ through the ``face'' given by equation $g_{i}=0$. Therefore, $\max_{i \in \II(z)} L_{f} g_{i} (z,u(0+)) \leq 0$, hence (\ref{us-boundary:eq}).
\end{proof}

\begin{rem} The subset of the boundary $G_{0}$ satisfying (\ref{us-boundary:eq}) is called the \emph{usable part of the constraint boundary} in reference to Isaacs' notion of usable part of a target\footnote{Note that, in our context, the boundary $G_{0}$ of the constraint (\ref{eq:state_const}) cannot be interpreted as a target since we do not \emph{a priori} want the trajectories to finish on $G_{0}$. On the contrary, $G_{0}$ may be considered as a target for the non admissible trajectories (starting in $\AC$) and the usable part of this target is given by the relation $\min_{u\in U_{1}} \max_{i \in \II(z)} L_{f} g_{i}(z,u) > 0$, opposite to (\ref{us-boundary:eq}).} \cite{Isaacs}. See also e.g. \cite{Aubin,Quincampoix_siam}.
\end{rem}

We now turn to the study of $\DAM$, the complement of $\DAO$ in the boundary $\DA$.
\begin{pr}\label{boundary:prop} Assume that (A1) to (A4) hold.
The boundary subset $\DAM$ is made of points $\bar{x}\in G_{-}$ for which there exists $\bar{u}\in \UU$ and an arc of integral curve $x^{(\bar{u},\bar{x})}$ entirely contained in $\DAM$ until it intersects $G_0$ at a point $x^{(\bar{u},\bar{x})}(\bar{t})$ such that $g(x^{(\bar{u},\bar{x})}(\bar{t}))\circeq 0$ for some (possibly infinite) $\bar{t}\in [0,+\infty)$.
\end{pr}
\begin{proof}
Let $\bar{x}\in \DAM$, therefore satisfying (\ref{DAopt}). In particular, there exists $\bar{u}\in \UU$ such that $\ds \sup_{t\in [0,\infty)}  \max_{i=1,\ldots, p}g_{i}(x^{(\bar{u},\bar{x})}(t)) = 0$, and
$\bar{t}\in [0, +\infty)$ such that $\ds \max_{i=1,\ldots, p}g_{i}(x^{(\bar{u},\bar{x})}(\bar{t})) = 0$ for the first time (i.e. $\ds \max_{i=1,\ldots, p}g_{i}(x^{(\bar{u},\bar{x})}(t)) < 0$ for all $t < \bar{t}$). 
Then for an arbitrary $t_0 \in [0,\bar{t}[$, the point $\xi = x^{(\bar{u},\bar{x})}(t_0) \in G_{-}$ satisfies, by a standard dynamic programming argument (since $x^{(\bar{u},\xi)}(t)=x^{(\bar{u},\bar{x})}(t_0+t)$ for all $t\geq 0$), $\ds \min_{u\in \UU} \sup_{t\in [0,\infty)}  \max_{i=1,\ldots, p}g_{i}(x^{(u,\xi)}(t)) = 0$. It follows that $\xi\in \DAM$ and, therefore, the whole arc of integral curve starting from $\bar{x}\in G_{-}$ is entirely contained in $\DAM$ until it intersects\footnote{If $\bar{t} = +\infty$, the whole integral curve $x^{(\bar{u},\bar{x})}$ remains in $G_{-}$ and satisfies $\lim_{t\rightarrow +\infty} x^{(\bar{u},\bar{x})}(t) \in G_{0}$.} $G_0$ at time $\bar{t}$.
\end{proof}

\begin{rem}
Note that the arc of integral curve $x^{(\bar{u},\bar{x})}$ of Proposition~\ref{boundary:prop} may remain in $G_{0}$ after $\bar{t}$ for some time and then return to $G_{-}$.
\end{rem}

\begin{cor}\label{bar-sem-cor}
From any point on the boundary $\DAM$, there cannot exist a trajectory penetrating the interior of $\AAA$, denoted by $\Int(\AAA)$, before leaving $G_{-}$.
\end{cor}
\begin{proof}
Let $\bar{x}\in \DAM$. According to Proposition \ref{boundary:prop}, there exists  $\bar{u} \in \UU$ such that $x^{(\bar{u},\bar{x})}(t)\in \DAM$ for all $t\in [0, \bar{t}[$, where $\bar{t}$ is the first time such that $x^{(\bar{u},\bar{x})}(\bar{t}) \in G_{0}$. Given $\tau \in [0, \bar{t}[$, we denote  $\xi=x^{(\bar{u},\bar{x})}(\tau)$. Assume, by contradiction, that there exists $\tilde{u} \in \UU$ and $\sigma \in [0,\ee[$, with $\ee >0$ small enough so that $x^{(\tilde{u},\xi)}(\tau+t) \in G_{-}$ for all $t\in [0, \ee[$ and that $\zeta \triangleq x^{(\tilde{u},\xi)}(\tau + \sigma) \in \Int(\AAA)$. As a consequence of (\ref{AAAopt}) and (\ref{DAopt}), there exists $v\in \UU$ such that $\ds \sup_{t\in [0,\infty)}\max_{i=1,\ldots,p} g_{i}(x^{(v,\zeta)}(\tau + \sigma + t)) < 0$. Setting $\tilde{v}= \bar{u} \Join_{\tau} \tilde{u} \Join_{\tau+\sigma}v$, we easily verify that $\ds \sup_{t\in [0,\infty)}\max_{i=1,\ldots, p} g_{i}(x^{(\tilde{v},\bar{x})}(t)) <0$, which implies, again by (\ref{AAAopt}) and (\ref{DAopt}), that $\bar{x} \in \Int(\AAA)$, hence contradicting the fact that $\bar{x}\in \DAM$. We thus conclude that no integral curve starting in $\DAM$ can penetrate the interior of $\AAA$ before leaving $G_{-}$.
\end{proof}

\section{Ultimate tangentiality conditions}\label{sec:UltTangCond}

As concluded from Proposition~\ref{boundary:prop}, the boundary subset $\DAM$ is made of integral curves intersecting $G_0$. In this section, we provide a precise characterisation of how this intersection occurs. We will first deal with the particular case of a scalar state constraint and then we extend the result to the general multi-dimensional constraint case.

\subsection{The scalar state constraint case}
We  consider here the case $p=1$ in (\ref{eq:state_const}), hence the maximum over $\{1,\ldots, p\}$ in all previous expressions is not necessary.
\begin{pr}\label{ult-tan-1d-pr}
There exists a point $z= x^{(\bar{u},\bar{x})}(\bar{t})\in \cl(\DAM)\cap G_0$, i.e. such that $g(x^{(\bar{u},\bar{x})}(\bar{t}))$ $= 0$ for some finite time $\bar{t}\geq 0$ and satisfying
\begin{equation}\label{ult-tan-eq}
\min_{u\in U_{1}} L_{f}g (z,u) = 0.
\end{equation}
In other words, at such an intersection point $z$, there exists $u\in U_{1}$ such that the vector field $f(z,u)$ is  tangent to $G_0$ at $z$,
and $L_{f} g(z,v) \geq 0$ for all $v \in U_{1}$, i.e. the vector field $f(z,v)$ points outwards of $G_{-}$ at $z$ for all $v\in U_{1}$.
\end{pr}
\begin{proof}
Consider $\bar{x} \in \DAM \subset G_{-}$ and $\bar{u}\in \UU$ such that the integral curve $x^{(\bar{u},\bar{x})}(t) \in \DAM$ for all $t$ in some time interval until it reaches $G_0$ (as in Proposition~\ref{boundary:prop}) and, consider further that $\bar{u}\in \UU$ is prolonged in such a way that the integral curve remains in $G$ for all times (which is possible by the definition and closedness of $\AAA$).

Let us introduce a control variation $u_{\kappa,\ee}$ of $\bar{u}$ of the form (\ref{u-var-eq}), where $\tau$ is an arbitrary Lebesgue point of $\bar{u}$ before the integral curve $x^{(\bar{u},\bar{x})}$ intersects $G_0$.

Since $\bar{x} \in \DAM$, there exists an open set ${\mathcal O} \subset \RR^n$ such that $\bar{x}+\ee h \in \AC$ for all $h \in {\mathcal O}$ and $\Vert h\Vert \leq H$, with $H$ arbitrarily small, and all $\ee$ sufficiently small.
Therefore, by definition of $\AC$, the integral curve $t\mapsto x^{(u_{\kappa,\ee},\bar{x}+\ee h)}(t)$ must leave  $G_{-}$ and crosses $G_{0}$ at some $t(\kappa,\ee,h) < \infty$. Moreover, according to Lemma~\ref{approx-lem},  the family of integral curves $x^{(u_{\kappa,\ee},\bar{x}+\ee h)}$, indexed by $\ee$, converges to $x^{(\bar{u},\bar{x})}$ as $\ee \rightarrow 0$, uniformly with respect to $t$, $\kappa$ and $h$. 
Consequently, with the notations of Section~\ref{perturb-subsec-append-b}, the set 
$$\{ t(\kappa,\ee,h) : \kappa \in U_{1} \times [0,T] \times [0,L], \ee \in [0,\ee_{0}], \Vert h \Vert \leq H, h\in {\mathcal O} \}$$ 
is uniformly bounded and there exists a continuous scalar function $\rho$ such that 
$\lim_{\ee \rightarrow 0}\rho(\ee) = 0$ and 
$$ -\rho(\ee) \leq g(x^{(\bar{u},\bar{x})}(t(\kappa,\ee,h))) \leq \rho(\ee)$$ 
for all $\ee$ and $\Vert h\Vert$ sufficiently small and all $\kappa$. As a consequence of the uniform convergence of $x^{(u_{\kappa,\ee},\bar{x}+\ee h)}$ to $x^{(\bar{u},\bar{x})}$ and the continuity of the function $t \mapsto g\circ x^{(\bar{u},\bar{x})}(t)$, there exists some finite $\bar{t}$ satisfying $\lim_{\ee\rightarrow 0} g(x^{(u_{\kappa,\ee},\bar{x}+\ee h)}(t(\kappa,\ee, h) ))= 0 = g(x^{(\bar{u},\bar{x})}(\bar{t}))$. We denote $z \triangleq x^{(\bar{u},\bar{x})}(\bar{t}) \in G_0$.

By the definition of $t(\kappa,\ee, h)$ we have $g(x^{(u_{\kappa,\ee},\bar{x}+\ee h)}(t(\kappa,\ee, h) ))= 0$, and by the definition of $\bar{u}$ we have $g(x^{(\bar{u},\bar{x})}(t(\kappa,\ee, h))) \leq 0$. Thus, 
according to (\ref{approx-eq}) of Appendix~\ref{Append-B}, we have
$$\begin{aligned}
g(x^{(u_{\kappa,\ee},\bar{x}+\ee h)}(t(\kappa,\ee, h)))  &- g(x^{(\bar{u},\bar{x})}(t(\kappa,\ee, h))) \\&= \ee Dg(x^{(\bar{u},\bar{x})}(t(\kappa,\ee, h))) w(t(\kappa,\ee, h), \kappa,h) + O(\ee^2) \geq 0
\end{aligned}$$

Therefore, using (\ref{needle-eq}), after division by $\ee$,
$$\begin{aligned}
Dg(x^{(\bar{u},\bar{x})}&(t(\kappa,\ee, h)))\left[ \Phi^{\bar{u}}(t(\kappa,\ee, h),0)h \right.\\
&\left.+ \; l\Phi^{\bar{u}}(t(\kappa,\ee, h),\tau) \left( f(x^{(\bar{u},\bar{x})}(\tau), v) - f(x^{(\bar{u},\bar{x})}(\tau), \bar{u}(\tau)) \right)\right] + O(\ee)\geq 0.
\end{aligned}$$

Letting $\Vert h\Vert$ and $\ee$ tend to 0, dividing by $l$, and using again the uniform convergence, we readily get:
$$Dg(x^{(\bar{u},\bar{x})}(\bar{t}))\Phi^{\bar{u}}(\bar{t},\tau) \left( f(x^{(\bar{u},\bar{x})}(\tau), v) - f(x^{(\bar{u},\bar{x})}(\tau), \bar{u}(\tau)) \right) \geq 0$$
for all $v\in U_{1}$ and for all Lebesgue points $\tau \leq \bar{t}$.

Thus, if $\bar{t}$ is a Lebesgue point of $\bar{u}$, setting $\tau = \bar{t}$, we have, using the  Lie derivative notation,
$$
L_{f} g(x^{(\bar{u},\bar{x})}( \bar{t}),v) \geq
L_{f} g(x^{(\bar{u},\bar{x})}(\bar{t}),\bar{u}(\bar{t})) 
$$
for all $v\in U_{1}$, i.e. 
\begin{equation}\label{minLfg-eq}
L_{f} g(z, \bar{u}(\bar{t})) = \min_{v\in U_{1}} L_{f} g(z, v).
\end{equation}

Otherwise, if $\bar{t}$ is not a Lebesgue point of $\bar{u}$, choosing $\tau$ arbitrarily close to $\bar{t}$, since $\Phi^{\bar{u}}(\bar{t},\tau)$ is arbitrarily close to the identity matrix, the same conclusion holds true by replacing $\bar{u}(\bar{t})$ by the left limit $\bar{u}_{-}(\bar{t}) \triangleq \lim_{\tau \rightarrow \bar{t}, \tau \leq \bar{t}} \bar{u}(\tau)$. It suffices then to modify $\bar{u}$, without loss of generality, on the 0-measure set $\{\bar{t}\}$ by $\bar{u}_{-}(\bar{t})$ to obtain (\ref{minLfg-eq}).

Finally, since $z\in \DAO$, according to Proposition~\ref{pr:usable}, we have
$\min_{v} L_{f}g(z,v) \leq 0$.
But, since the mapping $t\mapsto g(x^{(\bar{u},\bar{x})}(t))$ is non decreasing in an interval $]\bar{t}-\delta, \bar{t}]$ with $\delta >0$ small enough, we also have 
$0 \leq L_{f}g(x^{(\bar{u},\bar{x})}(\bar{t}), \bar{u}_{-}(\bar{t}))= \min_{v} L_{f}g(z,v)$ and thus $\min_{v} L_{f}g(z,v) = 0$,
which achieves the proof of the proposition.
\end{proof}

\subsection{The multi-dimensional state constraint case}

We now consider the general case of multiple $p$ state constraints in (\ref{eq:state_const}).

\begin{pr}\label{ult-tan-nd-pr} 
There exists a point $z= x^{(\bar{u},\bar{x})}(\bar{t})\in \cl(\DAM)\cap G_0$, i.e. such that $g(x^{(\bar{u},\bar{x})}(\bar{t}))\circeq 0$ for some finite time $\bar{t}\geq 0$ and, with the notations introduced in Section~\ref{sec:ConsDynCon},
\begin{equation}\label{ult-tan-eq-mult}
\min_{u\in U_{1}} \max_{i \in \II (z)} L_{f} g_{i}(z,u) = 0.
\end{equation}
In other words, there exists a point $z$ at the intersection between $G_{0}$ and an arc of integral curve in the boundary $\DAM$, and $u\in U_{1}$ such that the vector field $f(z,u)$ is  tangent at $z$ to at least  one face defined by $g_{i}(x)=0$, for $i \in \II (z)$, and does not point outwards with respect to any face corresponding to $\II (z)$, i.e. $L_{f}g_{j}(z,u) \leq 0$ for all $j\in \II (z)$.
In addition, for all $v \in U_{1}$, the vector field $f(z,v)$ points outwards of $G_{-}$ at $z$, i.e. there exists $j\in \II(z)$ such that  $L_{f}g_{j}(z,v) \geq 0$.
\end{pr}

\begin{proof}
Consider $\bar{x} \in \DAM \subset G_{-}$ and a finite
$\bar{t}\geq 0$, obtained as the limit of the crossing times associated to a sequence of variations as in the proof of Proposition~\ref{ult-tan-1d-pr}, such that 
\begin{equation}\label{traj-DAM-eq}
\max_{i=1,\ldots, p}g_{i}(x^{(\bar{u},\bar{x})}(\bar{t})) = 0
\end{equation}   
and denote $z \triangleq x^{(\bar{u},\bar{x})}(\bar{t})$.

By the same argument as in the proof of Proposition~\ref{ult-tan-1d-pr}, each element of the perturbed sequence $x^{(u_{\kappa,\ee},\bar{x}+\ee h)}$ leaves $G_{-}$ through one of the faces of $G$,  say $i=i(\kappa,\ee, h)\in \II(z)$, of equation $g_{i(\kappa,\ee, h)}(x)=0$. Choosing a subsequence of $\{\kappa,\ee, h\}$, for $\ee$ and $h$ sufficiently small, such that $i(\kappa,\ee, h)=i_0$ for some fixed $i_0\in \II(z)$, we readily deduce from the proof of Proposition~\ref{ult-tan-1d-pr} that 
\begin{equation}\label{ult-tan-i0-eq}
\min_{u\in U_{1}} L_{f}g_{i_{0}} (z,u) = 0.
\end{equation}
Taking into account that $z\in \DAO$, the result of Proposition~\ref{pr:usable}, says that
$$0= \min_{u\in U_{1}} L_{f}g_{i_{0}} (z,u) \leq \min_{u\in U_{1}} \max_{i\in \II(z)} L_{f}g_{i} (z,u) \leq 0$$
which  immediately gives (\ref{ult-tan-eq-mult}).
\end{proof}

\section{The barrier equation}\label{sec:BarrEqn}

For topological reasons, we now restrict our attention to the part of the boundary $\DAM$ which intersects the closure of the interior of $\AAA$, denoted $\cl(\Int(\AAA))$,  assuming that $\Int(\AAA) \neq \emptyset$.

\begin{thm}\label{barrier:thm}
Under the assumptions of Proposition~\ref{boundary:prop}, every integral curve $x^{\bar{u}}$ on $\DAM \cap \cl(\Int(\AAA))$ and the corresponding control function $\bar{u}$, as in Proposition~\ref{boundary:prop},  satisfies the following necessary condition.

There exists a (non zero) absolutely continuous maximal solution $\lambda^{\bar{u}}$ to the adjoint equation
\begin{equation}\label{adjoint}
\dot{\lambda}^{\bar{u}}(t)= - \left( \frac{\partial f}{\partial x}(x^{\bar{u}}(t),\bar{u}(t))\right)^T\lambda^{\bar{u}}(t) , \quad \lambda^{\bar{u}}(\bar{t}) =  \left( Dg_{i^{\ast}}(z)\right)^{T}
\end{equation}
 such that
\begin{equation}\label{barriercond} \min_{u\in U_{1}} \left\{(\lambda^{\bar{u}}(t))^T f(x^{\bar{u}}(t),u)\right\}=(\lambda^{\bar{u}}(t))^T f(x^{\bar{u}}(t),\bar{u}(t))=0
\end{equation}
at every Lebesgue point $t$ of $\bar{u}$ (i.e. for almost all $t \leq \bar{t}$).

In (\ref{adjoint}), 
$\bar{t}$ denotes the time\footnote{Note that $\bar{t}$ can be chosen arbitrarily because of the time-invariance of the problem.} at which $z$ is reached, i.e. $x^{\bar{u}}(\bar{t})=z$, with $z\in G_0$ satisfying
$$g_{i}(z)=0, \quad i\in \II(z), \quad \min_{u\in U_{1}}  \max_{i \in \II (z)} L_f g_{i}(z,u) =  L_f g_{i^{\ast}}(z,u^{\ast}) = 0. $$

Moreover, $\lambda^{\bar{u}}(t)$ is  normal  to $\DAM \cap \cl(\Int(\AAA))$ at $x^{\bar{u}}(t)$  for almost every $t \leq \bar{t}$.
\end{thm}

Before proving this theorem, we provide several comments and we recall that the attainable set $X_{t}(\bar{x})$ from a point $\bar{x}\in \RR^{n}$ at time $t$ is the subset of $\RR^{n}$ defined by (\ref{attain-set}). 
As a direct consequence of Lemma~\ref{compact-lem} of  Appendix~\ref{Append-A}, $X_{t}(\bar{x})$ is compact for all finite $t$.
We denote by $\partial X_{t}(\bar{x})$ its boundary.

\begin{rem}
Theorem~\ref{barrier:thm} may be interpreted as follows:
Setting
\begin{equation}\label{Hamdef}
H(x,\lambda,u)= \lambda^{T}f(x,u)
\end{equation}
$\DAM \cap \cl(\Int(\AAA))$ is made of trajectories $t\mapsto x^{\bar{u}}(t)$, which are projections by
$$\pi: (x, \lambda, u) \mapsto~x= \pi (x, \lambda, u)$$
of trajectories of the triple $(x^{\bar{u}}, \lambda^{\bar{u}}, \bar{u})$, solution to the Hamiltonian system
\begin{equation}\label{hamilton}
\begin{array}{l}
\ds \dot{x}^{\bar{u}}(t)= \left(\frac{\partial H}{\partial \lambda}\right)^{T}(x^{\bar{u}}(t),\lambda^{\bar{u}}(t),\bar{u}(t)), \quad
\ds \dot{\lambda}^{\bar{u}}(t)=-\left( \frac{\partial H}{\partial x}\right)^{T}(x^{\bar{u}}(t),\lambda^{\bar{u}}(t),\bar{u}(t))\vspace{0.6em}\\
\ds H(x^{\bar{u}}(t),\lambda^{\bar{u}}(t),\bar{u}(t)) =  \min_{u\in U_{1}} H(x^{\bar{u}}(t),\lambda^{\bar{u}}(t),u)=0
\end{array}
\end{equation}
computed backwards in time from the final condition 
$$(x^{\bar{u}}(\bar{t}), \lambda^{\bar{u}}(\bar{t}), \bar{u}(\bar{t})) = (z,\left( Dg_{i^{\ast}}(z)\right)^{T},u^{\ast}) \in G_{0} \times \RR^{n} \times U_{1},$$ 
such that the set of equations
\begin{equation}\label{finalcond}
g_{i}(z)=0, \quad i\in \II(z), \quad \min_{u\in U_{1}} \max_{i \in \II(z)}  L_f g_{i}(z,u)=  L_f g_{i^{\ast}}(z,u^{\ast}) = 0
\end{equation}
admits a local solution.
\end{rem} 

\begin{figure}[h]
\begin{center}
\includegraphics[width=.6\columnwidth]{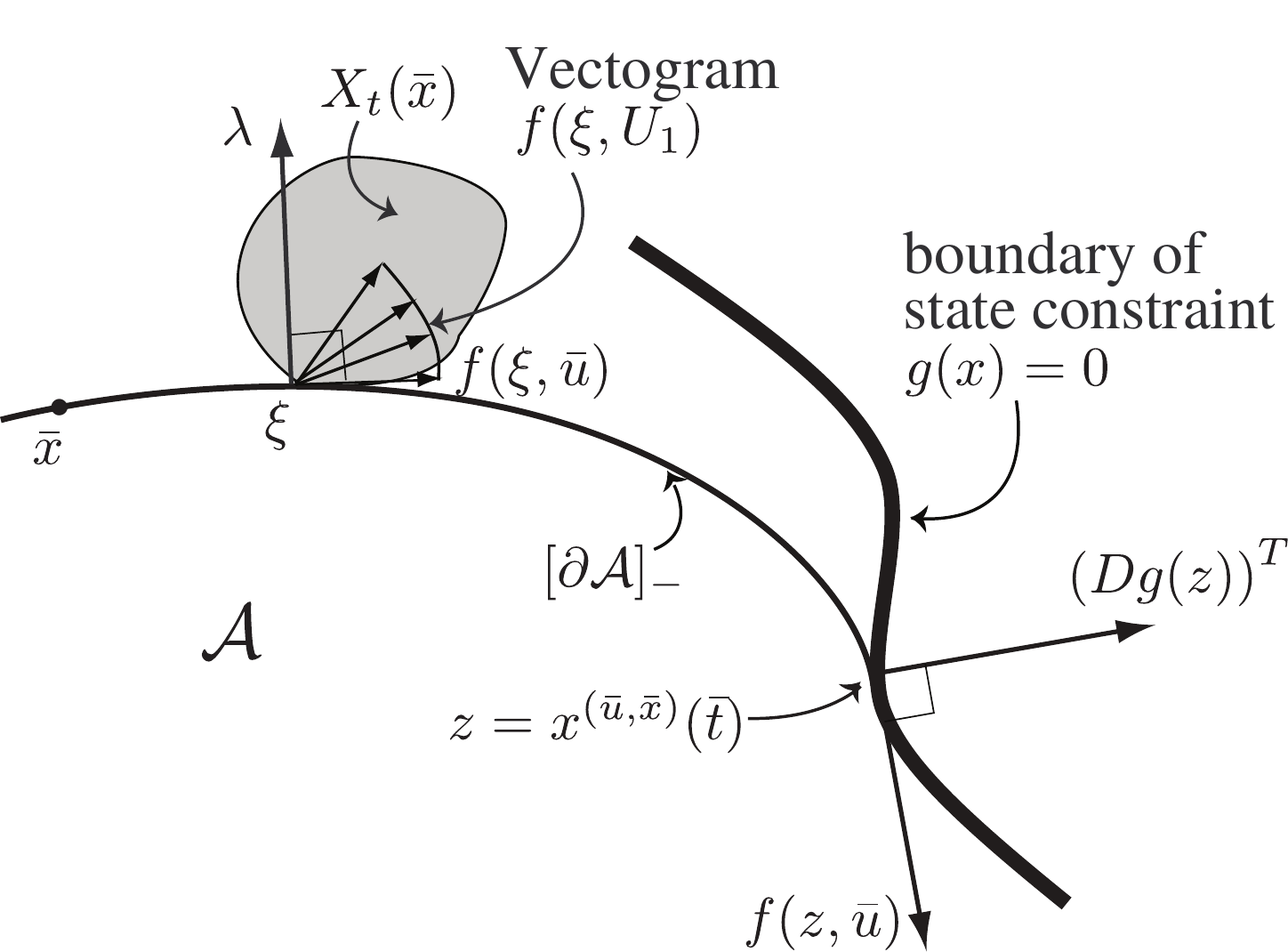}
\end{center}
\caption{Each point of the boundary, $\xi\in \DAM\cap \cl(\Int(\AAA))$, also belongs to the boundary of the attainable set, $\partial X_{t}(\bar{x})$, from $\bar{x}\in \DAM \cap \cl(\Int(\AAA))$ at every sufficiently small time $t$,  and the vectogram $f(\xi,U_{1})$ points outwards with respect to $\AAA$ and inwards with respect to $X_{t}(\bar{x})$. In other words, $\lambda$, the normal to $\DAM\cap \cl(\Int(\AAA))$, is such that $\lambda^{T} f(\xi,u) \geq 0$ for all $u\in U_{1}$.}\label{attfig}
\end{figure}

\begin{rem}
The necessary conditions (\ref{hamilton})-(\ref{finalcond}) may be compared to the one obtained by Isaacs in the context of barriers in differential games \cite{Isaacs} and, of course, to the Pontryagin maximum principle (PMP) of \cite{PBGM} in the context of optimal control. Note however that in our case, contrarily to those references, there is no a priori cost function to optimise (note that, even though the characterisation of the admissible set (\ref{AAAopt}) may be interpreted in terms of a minimisation problem, the cost function is not a standard one). Moreover, as far as the constraints cannot be directly interpreted as a separate player, there is no game involved.

However, it has been remarked since long (see e.g. \cite[Chapter 4, p. 239]{Lee_Markus},\cite[Section 10.2, p. 136]{Agrachev}) that the maximisation of the Hamiltonian in the PMP was in fact characterising the boundary of the attainable set from prescribed initial conditions. More precisely, at every boundary point of the attainable set, no vector of the tangent perturbation cone (see \cite{PBGM,Lee_Markus,Agrachev}), and in particular of the vectogram of admissible directions, can point outwards the attainable set, i.e. the scalar product of any such element of the perturbation cone with the normal vector to the boundary constituted by the adjoint is non positive (see Theorem~\ref{extrem:thm} in Appendix~\ref{Append-B}). In addition, the adjoint vector is obtained by parallel transport  of the final conditions, the so-called transversality conditions, along an optimal integral curve (Equation (\ref{adjoint-eta})). Here, we are in a similar situation, as depicted by Figure~\ref{attfig}, the only difference with respect to the PMP being that on $\DAM\cap \cl(\Int(\AAA))$ the adjoint vector $\lambda$ we consider is an outer normal to $\DAM\cap \cl(\Int(\AAA))$ and therefore a normal pointing inwards the attainable set, opposite to the one considered in the PMP, hence the minimum replacing the maximum with respect to $u$ (see Equation (\ref{barriercond})).
\end{rem}

We need the following proposition.

\begin{pr}\label{att-boundary:prop}
Let $\bar{x} \in \DAM\cap \cl(\Int(\AAA))$ and $\bar{u}\in \UU$ as in Proposition~\ref{boundary:prop}, i.e. such that $x^{(\bar{u},\bar{x})}(t)\in \DAM\cap \cl(\Int(\AAA))$ for all $t\in [0, \bar{t}[$, where $\bar{t}$ is the first time such that 
$$\max_{i=1,\ldots,p} g_{i}(x^{(\bar{u},\bar{x})}(\bar{t}))=0.$$ 
Then $x^{(\bar{u},\bar{x})}(t)\in \partial X_{t}(\bar{x})$ for all sufficiently small $0\leq t < \bar{t}$.
\end{pr}
\begin{proof}
We first prove that $X_{t}(\bar{x}) \subset \cl (\AC)$ for all sufficiently small $t$. Since $\bar{x}\in G_{-}$, by the  continuity of the attainable set $X_{t}(\bar{x})$ with respect to $t$ (see \cite{Lee_Markus}) there exists a small enough time interval in which $X_{t}(\bar{x}) \subset G_{-}$. It follows from Corollary~\ref{bar-sem-cor} that $X_{t}(\bar{x}) \cap \Int(\AAA) = \emptyset$, thus, $X_{t}(\bar{x}) \subset \cl(\AC)$.

Thus, by complementarity, $\Int(\AAA)\subset X_{t}(\bar{x})^{\mathsf  C}$, and consequently $\cl(\Int(\AAA))\subset \cl(X_{t}(\bar{x})^{\mathsf  C})$.
Therefore
$x^{(\bar{u},\bar{x})}(t) \in \DAM\cap \cl(\Int(\AAA))  \subset \cl(X_{t}(\bar{x})^{\mathsf  C})$. Since, by definition,  we also have  $x^{(\bar{u},\bar{x})}(t)\in X_{t}(\bar{x})$, the conclusion follows from the fact that $\partial X_{t}(\bar{x})= X_{t}(\bar{x}) \cap \cl(X_{t}(\bar{x})^{\mathsf  C})$.
\end{proof}

The proof of Theorem~\ref{barrier:thm}, is based on the necessary condition for a point of the state space to belong to the boundary of an attainable set given in Theorem~\ref{extrem:thm} of Appendix~\ref{Append-B}  (\cite[Theorem 3, Chapter 4, p. 254]{Lee_Markus} or \cite[Theorem 12.1, p.164 and Theorem 12.4 p. 178]{Agrachev}. See also \cite{Clarke,Vinter}).

\paragraph{Proof of Theorem~\ref{barrier:thm}.}

Let $\bar{x}$ and $\bar{u}$ be as in Proposition~\ref{att-boundary:prop}. Denote $x_{1}=x^{(\bar{u},\bar{x})}(t_{1})$ and $x_{2} =x^{(\bar{u},x_{1})}(t_{2})$ for sufficiently small $t_{1}$ and $t_{2}$ such that $x_{1}\in \partial X_{t_{1}}(\bar{x})$ and $x_{2}\in \partial X_{t_{2}}(x_{1})$ as in Proposition~\ref{att-boundary:prop}.
By Theorem~\ref{extrem:thm}, there exists an absolutely continuous maximal solution $\eta^{\bar{u}}_{1}$ (resp. $\eta^{\bar{u}}_{2}$) satisfying Equations (\ref{adjoint-eta})--(\ref{barriercond-eta}) on the interval $[0,t_{1}]$ (resp. $[t_{1},t_{2}]$). By the homogeneity of Equation (\ref{adjoint-eta}), the end-point conditions can be chosen so that $\eta^{\bar{u}}_{1}(t_{1}) =\eta^{\bar{u}}_{2}(t_{1})$, thus achieving the same constant in (\ref{barriercond-eta}).
It follows that there exists $\eta^{\bar{u}}$ defined on $[0,t_{2}]$ such that $\eta^{\bar{u}}= \eta^{\bar{u}}_{1}$ on $[0, t_{1}]$ and $\eta^{\bar{u}}= \eta^{\bar{u}}_{2}$ on $[t_{1}, t_{2}]$. By the same argument, the solution $\eta^{\bar{u}}$ can be extended to any subinterval of $[0, \bar{t}]$.
Denoting $\lambda^{\bar{u}} = - \eta^{\bar{u}}$, it is immediately seen that $\lambda^{\bar{u}}$ satisfies (\ref{adjoint-eta}) and that the $\max$ operator in (\ref{barriercond-eta}) is replaced by $\min$. 

Moreover, the final condition of $\lambda^{\bar{u}}$ at $\bar{t}$ is given by Proposition~\ref{ult-tan-nd-pr}, hence
 (\ref{adjoint}), which implies that the constant of the right-hand side of (\ref{barriercond-eta}) is equal to 0.
We have thus proven (\ref{adjoint}) and (\ref{barriercond}), which achieves the proof of Theorem~\ref{barrier:thm}. \hspace*{\fill}\bBox

\section{Examples}\label{sec:Examples}
We present a number of examples to illustrate the use of the results of the paper.
\subsection{Linear spring}

Consider a system consisting of a mass and a spring, governed by the differential equation
$m\ddot{y}+b\dot{y}+ky=u$, where $m$ is the mass, $y$ the displacement, $b$ the (linear) friction coefficient, $k$ the
spring constant and $u$ the force applied to the mass. We consider a state constraint as in~\eqref{eq:state_const}
of the form $y(t) \leq \bar{x}_1$, where $\bar{x}_1$ is a constant position that must not be exceeded, and a constraint in the input as
in~\eqref{eq:input_constraint}, i.e., $|u(t)| \leq 1$. By introducing the variables $x_1 = y$, $x_2 = \dot{y}$,
the system and constraints can be written as:
$$\left( \begin{array}{c}\dot{x}_{1}\\\dot{x}_{2}\end{array}\right) =
\left( \begin{array}{cc}0&1\\-\frac{k}{m}&-\frac{b}{m}\end{array}\right)
\left( \begin{array}{c}x_{1}\\x_{2}\end{array}\right)
+ \left( \begin{array}{c}0\\\frac{1}{m}\end{array}\right) u, \quad |u| \leq  1, \quad x_1 - \bar{x}_1 \leq  0.
$$
From the ultimate tangentiality condition on $G_0$ [i.e., on $(x_1,x_2) = (\bar{x}_1,x_2)$, $x_2 \in \RR$], given 
by~\eqref{ult-tan-eq}, since $Dg(x)=(1,0)$, we have:
\begin{equation*}
0= \min_{|u|\leq 1} \left\{Dg(x) f(x,u)\right\}=\min_{|u|\leq 1} \left\{x_2\right\}=x_2.
\end{equation*}
Hence, the point $(\bar{x}_1,0)$ of $G_{0}$  is the endpoint of the trajectory
in $\cl(\DAM)$ that arrives tangentially to $G_{0}$ (see Propositions~\ref{boundary:prop} and~Ê\ref{ult-tan-1d-pr}).

The rest of the points belonging to the barrier $\DAM$ are given by equation~\eqref{barriercond},
$$\min_{|u|\leq 1} \left\{\lambda^T f(x,u)\right\}=0,$$ 
and satisfy equation~\eqref{adjoint}, that is,
\begin{equation}
\label{adjoint_example}
\dot{\lambda}=
      \left(
       \begin{array}{c}
         \dot{\lambda}_1 \\
         \dot{\lambda}_2 \\
       \end{array}
     \right)=  - \left( \frac{\partial f}{\partial x}(x,u)\right)^T\lambda =
     \left(
       \begin{array}{cc}
         0  & \frac{k}{m} \\
         -1 & \frac{b}{m} \\
       \end{array}
     \right)
     \left(
       \begin{array}{c}
         \lambda_1 \\
         \lambda_2 \\
       \end{array}
     \right), \;
    \left(
       \begin{array}{c}
         \lambda_1(\bar{t}) \\
         \lambda_2(\bar{t}) \\
       \end{array}
     \right)= 
     \left(
       \begin{array}{c}
         1 \\
         0 \\
       \end{array}
     \right),
\end{equation}
where $\bar{t}$ is the time of arrival at the endpoint $(\bar{x}_1,0)$. The trajectory $\lambda(t)$ of the adjoint system can thus be obtained
by backward integration of~\eqref{adjoint_example}. In addition, from equation~\eqref{barriercond} we have:
\begin{equation}\label{barrier_example}
    \min_{|u|\leq 1} \left\{\lambda_1x_2+\lambda_2(-\frac{k}{m}x_1-\frac{b}{m}x_2+\frac{1}{m}u)\right\}=0,
\end{equation}
from where we deduce that, on the barrier, $u=-\mathrm{sgn}(\lambda_2)$ and that the barrier points are described by:
\begin{equation}
\label{system_example}
    \left(
       \begin{array}{c}
         \dot{x}_1 \\
         \dot{x}_2 \\
       \end{array}
     \right)=
     \left(
       \begin{array}{cc}
         0  & 1 \\
         -\frac{k}{m} & -\frac{b}{m} \\
       \end{array}
     \right)
     \left(
       \begin{array}{c}
         x_1 \\
         x_2 \\
       \end{array}
     \right)-
     \left(
       \begin{array}{c}
         0 \\
         \frac{1}{m} \\
       \end{array}
     \right)\mathrm{sgn}(\lambda_2), \quad
     \left(
       \begin{array}{c}
         x_1(\bar{t}) \\
         x_2(\bar{t}) \\
       \end{array}
     \right)=
     \left(
       \begin{array}{c}
         \bar{x}_1 \\
         0 \\
       \end{array}
     \right).
\end{equation}
The trajectory $x(t)$ determining the barrier can thus be obtained
by backward integration of~\eqref{system_example}. The resulting trajectory that gives the barrier
of the admissible set $\AAA$ is shown in Figure~\ref{fig:Figure_example_1}
\begin{figure}[thpb]
\begin{center}
\includegraphics[width=.6\columnwidth]{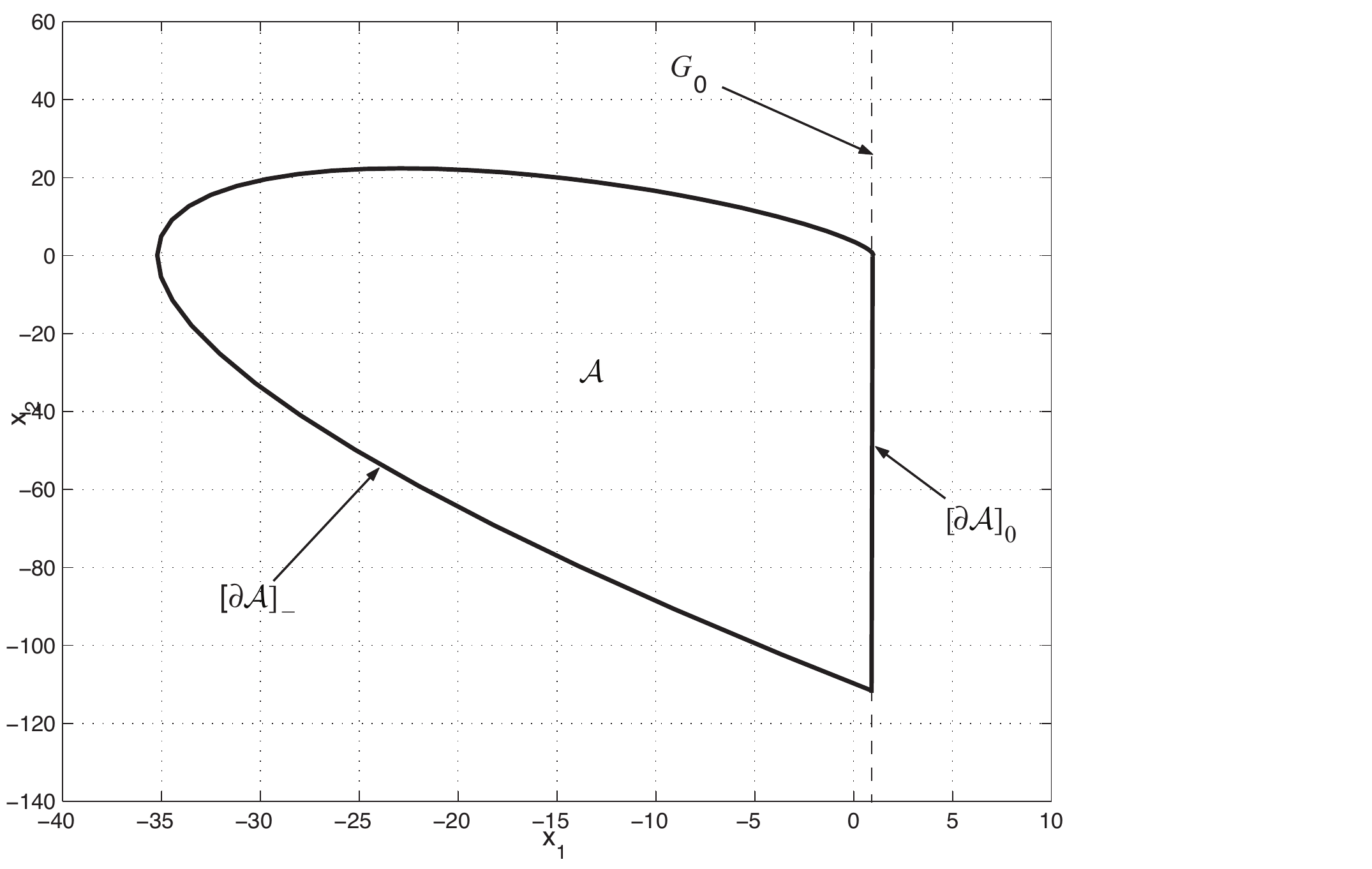}
\caption{Admissible set $\AAA$ and barrier for the linear mass-spring example}
\label{fig:Figure_example_1}
\end{center}
\end{figure}
for the numerical values $m=1$, $k=2$, $b=2$ and $\bar{x}_1=1$.

\subsection{Nonlinear spring}
We next consider a nonlinear version of the spring, in which the force exerted by the spring is given
by $k(x_1+x_1^3)$ instead of the linear force $kx_1$ and where, as before, $x_1=y=$ displacement (this situation is commonly referred to as
\emph{hardening spring}). From an analysis almost identical to the one performed previously for the linear spring, we conclude that the
barrier points can be obtained from the following system and adjoint equations:
\begin{equation}\label{nonlinear}
\left\{ \begin{array}{lcl}
   \dot{\lambda}_1 &=& \frac{k}{m}(1+3x_1^2)\lambda_2\\
   \dot{\lambda}_2 &=& -\lambda_1 + \frac{b}{m} \lambda_2
 \end{array}\right. , \quad
 \left\{ \begin{array}{lcl}
   \dot{x}_1 &=& x_2\\
   \dot{x}_2 &=& - \frac{k}{m}(x_1+x_1^3)-\frac{b}{m} x_2 -\frac{1}{m} \mathrm{sgn}(\lambda_2)
 \end{array}\right.
\end{equation}
with endpoint given by $(\lambda_1(\bar{t}),\lambda_2(\bar{t}),x_1(\bar{t}),x_2(\bar{t}))=(1,0,\bar{x}_1,0)$.
The trajectory determining the barrier can thus be obtained
by backward integration of system~\eqref{nonlinear} starting from the endpoint.
The resulting trajectory that gives the barrier
of the admissible set $\AAA$ is shown in Figure~\ref{fig:Figure_example_2}
\begin{figure}[thpb]
\begin{center}
\includegraphics[width=.6\columnwidth]{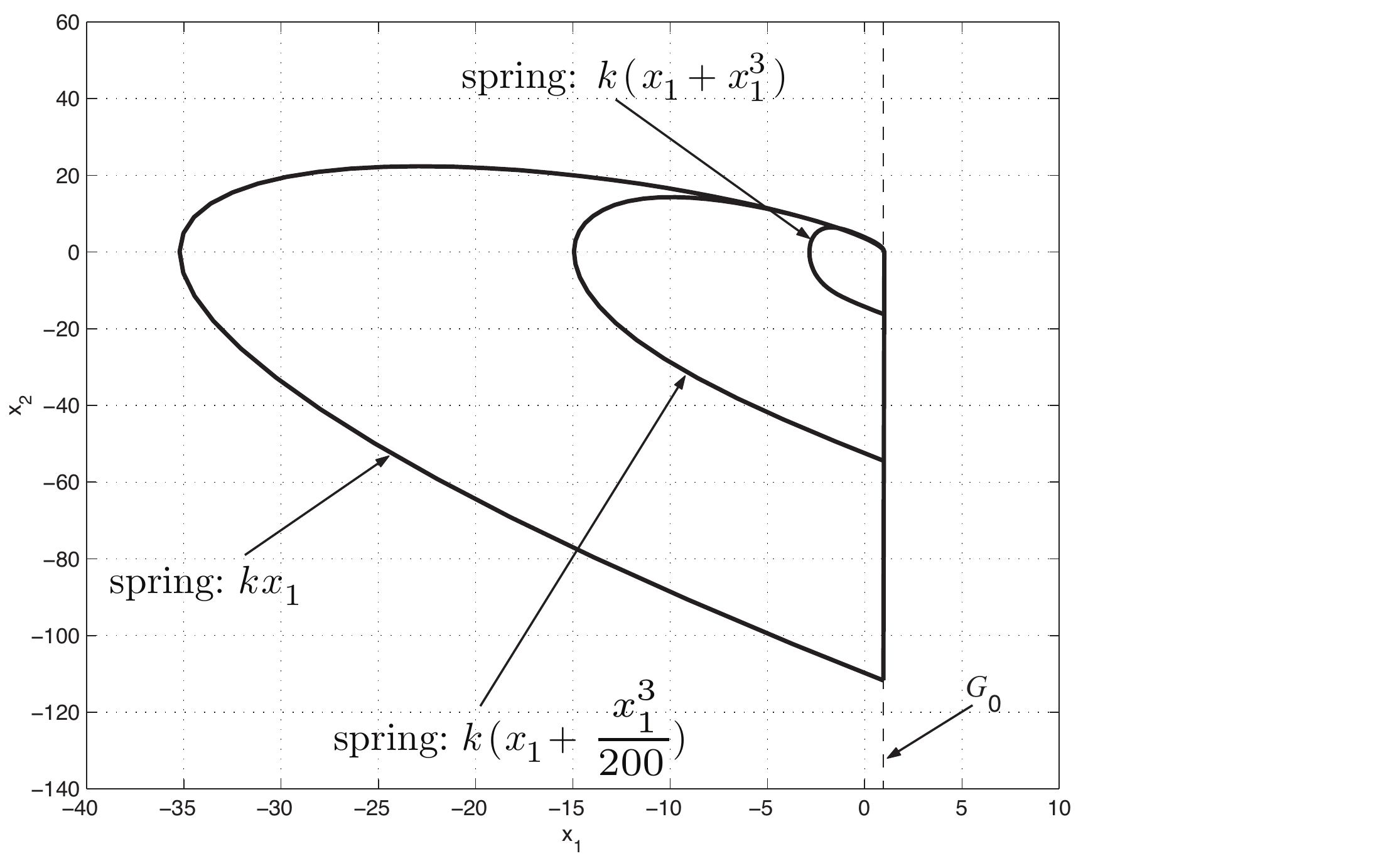}
\caption{Admissible sets $\AAA$ and barriers for the non-linear mass-spring examples}
\label{fig:Figure_example_2}
\end{center}
\end{figure}
for the same numerical values as before. For comparison purposes, also shown in the figure
are the cases of the linear spring $kx_1$ and an intermediate case $k(x_1+\frac{x_1^3}{200})$.
Note, from the figure, the considerable reduction in the size of the admissible set
due to the effect of the \emph{hardening spring}; namely, a hardening spring is able
to store more potential energy and, hence, it can surpass the position constraint if started
from initial conditions that would not cause constraint violations in the case of a linear spring.

\subsection{A nonlinear academic example}
We consider the 2-dimensional single input system:   
\begin{equation}\label{sys:ex}
\begin{array}{l}
\dot{x}_{1}= 1-x_{2}^{2} \\
\dot{x}_{2}=u
\end{array}
\end{equation}
with input constraint 
\begin{equation}\label{inputcons:ex}
\vert u \vert \leq 1
\end{equation}
and state constraint 
\begin{equation}\label{statecons:ex}
\underline{a} \leq x _{1} \leq \overline{a}, \quad \underline{a}, \overline{a}\in \RR, \quad \underline{a} < \overline{a}.
\end{equation}
Denoting as usual $x= (x_{1},x_{2})$, we have $G= \{ x \in \RR^{2}: \underline{a} \leq x _{1} \leq \overline{a}\}$. 
We set $\underline{g}(x) \triangleq  \underline{a} - x _{1}$ and $\overline{g}(x) \triangleq  x _{1}-\overline{a}$. Clearly, $G$ is given by the set of inequalities
$$\underline{g}(x) \leq 0, \qquad  \overline{g}(x) \leq 0.$$

The boundary $G_{0}$ of $G$ is thus made of the union of the sets $\underline{a} - x _{1} = 0$ and $x _{1}-\overline{a} = 0$. Since the latter sets define two disjoint 1-dimensional smooth manifolds (parallel straight lines), a normal vector to $G_{0}$  is given by $D\underline{g}(x)= (-1,0)$ if $x_{1} = \underline{a}$ and $D\overline{g}(x)= (1,0)$ if $x_{1} = \overline{a}$. We thus consider the two cases, where the active constraint is $x_{1} = \underline{a}$ or $x_{1} = \overline{a}$, separately.

\subsubsection{The case $x_{1} = \underline{a}$}
The ultimate tangentiality condition (\ref{ult-tan-eq-mult}) reads 
$$\min_{\vert u\vert \leq 1} \{-1(1-x_{2}^{2}(\underline{t})) +0\cdot u\}  = 0$$
where $\underline{t}$ is a time instant where an arc of the system integral curve contained in $\DAM$ intersects $\underline{G}_{0} \triangleq \{ x \in \RR^{2} : x_{1} = \underline{a}\}$. We readily get $x_{2}^{2}(\underline{t}) = 1$.
For convenience, we denote $x_{2}(\underline{t}) = \underline{x}_{2}$ and we have $\underline{x}_{2} = \pm 1$.

The Hamiltonian (\ref{Hamdef}) is $H(x,\lambda,u) \triangleq \lambda_{1}(1-x_{2}^{2}) + \lambda_{2}u$, where $\lambda = (\lambda_{1},\lambda_{2})$ is the adjoint state, and the adjoint equation is
$$\dot{\lambda}_{1} = 0, \qquad \dot{\lambda}_{2}= 2x_{2}\lambda_{1}$$
with final condition at time $\underline{t}$ 
$$\lambda_{1}(\underline{t})= -1, \qquad \lambda_{2}(\underline{t})= 0$$
The adjoint is thus given by $\lambda_{1}(t) \equiv -1$ for all $t \leq \underline{t}$ and 
$$\dot{\lambda}_{2}(t)= -2x_{2}(t), \qquad \lambda_{2}(\underline{t})= 0.$$
Moreover, we must have, according to (\ref{hamilton}),
$$\min_{\vert u \vert \leq 1} H(x(t),\lambda(t),u) \triangleq H(x(t),\lambda(t),\underline{u}(t)) = 0 , \quad a.e.~t \leq \underline{t}$$
or $\underline{u}(t)= - \sgn(\lambda_{2}(t))$ and $-1 + x_{2}^{2}(t) = \vert \lambda_{2}(t)\vert$. Since the r.h.s. of the latter expression is non negative, we deduce that either $x_{2}(t) \geq 1$ or $x_{2}(t) \leq -1$.

Setting $\underline{\lambda}_{2} \triangleq \sgn(\lambda_{2}(t))$ for $t$ in an interval $]t_{1},\underline{t}[$ to be determined,
the arc of curve of $\DAM$ arriving at $(\underline{a}, \underline{x}_{2})$ is given by
$$ \dot{x}_{1}(t)= 1 - x_{2}^{2}(t), \qquad \dot{x}_{2}(t) = - \underline{\lambda}_{2}$$
with 
$$x_{1}(\underline{t})= \underline{a}, \qquad x_{2}(\underline{t}) = \underline{x}_{2}$$
and with 
$$\dot{\lambda}_{2}(t) = - 2x_{2}(t), \qquad \lambda_{2}(\underline{t}) = 0.$$

We readily get:
\begin{equation}\label{x2solunder:ex2}
x_{2}(t) =  \underline{x}_{2} - \underline{\lambda}_{2}(t - \underline{t})
\end{equation}
and, since $\underline{x}_{2} = \pm 1$ and  $\underline{\lambda}_{2} = \pm 1$,
$1 - x_{2}^{2}(t) = 2  \underline{x}_{2} \underline{\lambda}_{2} (t - \underline{t}) - (t - \underline{t})^{2}$, and we have
\begin{equation}\label{x1solunder:ex2}
x_{1}(t) =  \underline{a} - \frac{1}{3} (t - \underline{t})^{3} +  \underline{x}_{2} \underline{\lambda}_{2} (t - \underline{t})^{2} =  \underline{a} - \frac{1}{3} (t - \underline{t})^{2} \left( t - \underline{t} - 3\underline{x}_{2} \underline{\lambda}_{2} \right)
\end{equation}
and, according to $\underline{\lambda}_{2}^{2} = 1$,
\begin{equation}\label{lambda2solunder:ex2}
\lambda_{2}(t) = -2 \underline{x}_{2}(t - \underline{t}) + \underline{\lambda}_{2} (t - \underline{t})^{2} =
 \underline{\lambda}_{2} (t - \underline{t}) \left( t - \underline{t} - 2\underline{x}_{2} \underline{\lambda}_{2} \right).
 \end{equation}
Since $x_{1}(t)$, given by (\ref{x1solunder:ex2}), has to satisfy the constraint $\underline{a} - x_{1}(t) \leq 0$, we have
$$\underline{a} - x_{1}(t) = \frac{1}{3} (t - \underline{t})^{2} \left( t - \underline{t} - 3\underline{x}_{2} \underline{\lambda}_{2} \right) \leq 0$$
$\forall t \in ]t_1,\underline{t}[$,
which, together with the fact that $\frac{1}{3} (t - \underline{t})^{2}  \geq 0$, yields
$t - \underline{t} - 3\underline{x}_{2} \underline{\lambda}_{2} \leq 0$, or
$$0 = \sup_{t_1 < t < \underline{t}} \frac{t - \underline{t}}{3} \leq \underline{x}_{2} \underline{\lambda}_{2}.$$
Thus, since $\underline{x}_{2} = \pm 1$ and  $\underline{\lambda}_{2} = \pm 1$, we get $\underline{x}_{2} \underline{\lambda}_{2} =  1$ and, 
$$\underline{x}_{2} = \underline{\lambda}_{2}.$$
Thus, (\ref{x2solunder:ex2}), (\ref{x1solunder:ex2}), (\ref{lambda2solunder:ex2}) read
\begin{equation}\label{solunder:ex2}
\begin{aligned}
x_{2}(t) &= - \underline{\lambda}_{2}(t - \underline{t} - 1)\\
x_{1}(t) &=  \underline{a} - \frac{1}{3} (t - \underline{t})^{2} \left( t - \underline{t} - 3 \right)\\
\lambda_{2}(t) &=  \underline{\lambda}_{2} (t - \underline{t}) \left( t - \underline{t} - 2 \right).
\end{aligned}
\end{equation}

Finally, eliminating $t - \underline{t}$ from the first equation of (\ref{solunder:ex2}), we get
$$t - \underline{t} = 1 - \frac{x_{2}}{\underline{\lambda}_{2}}$$
and
$$x_{1} =  \underline{a} + \frac{1}{3} \left(1 - \frac{x_{2}}{\underline{\lambda}_{2}}\right)^{2} \left( 2 + \frac{x_{2}}{\underline{\lambda}_{2}} \right).$$
Thus, if $\underline{\lambda}_{2}=1$ (which implies $\underline{u} = -1$):
\begin{equation}\label{branch1:ex2}
x_{1} =  \underline{a} + \frac{1}{3} \left( 1 - x_{2} \right)^{2} \left( 2 + x_{2} \right)
\end{equation}
an expression valid for $x_{2} \geq 1$ (since $t\leq \underline{t}$).

Now, if $\underline{\lambda}_{2}=-1$ (thus $\underline{u}=+1$):
\begin{equation}\label{branch2:ex2}
x_{1} =  \underline{a} + \frac{1}{3} \left( 1 + x_{2} \right)^{2} \left( 2 - x_{2} \right)
\end{equation}
an expression valid for $x_{2} \leq -1$ (since $t\leq \underline{t}$).

\subsubsection{The case $x_{1} = \overline{a}$}
The only difference with the previous case is that $\lambda_{1}(t)\equiv +1$ (the first component of the differential of $\overline{g}$ at $x_{1} = \overline{a}$) for $t\leq \overline{t}$,  where $\overline{t}$ is a time instant where an arc of the system integral curve contained in $\DAM$ intersects the line $x_{1}=\overline{a}$.
The ultimate tangentiality condition reads here $\min_{\vert u \vert \leq 1} \{1-\overline{x}_{2}^{2} + 0.u\} = 0$,
or $\overline{x}_{2} = \pm 1$. 

The Hamiltonian $H(x,\lambda,u) = \lambda_{1} (1 - x_{2}^{2}) + \lambda_{2} u$ is unchanged, as well as the argument of its minimum with respect to $u$, denoted by $\overline{u}(t)$, $\overline{u}(t) = - \sgn(\lambda_{2}(t))$. The adjoint equation, with  $\lambda_{1}(t)\equiv +1$, reads
$$\dot{\lambda}_{2}(t) =  2 x_{2}(t), \qquad \lambda_{2}(\overline{t}) = 0.$$

Setting $\overline{\lambda}_{2} \triangleq \sgn(\lambda_{2}(t))$ for $t$ in an interval $]t_{2},\overline{t}[$,
the integration of the system gives
$$\begin{aligned}
x_{2}(t) &= \overline{x}_{2} - \overline{\lambda}_{2} (t-\overline{t})\\
x_{1}(t) &= \overline{a} - \frac{1}{3} (t-\overline{t})^{2} \left( t-\overline{t} - 3\overline{x}_{2}\overline{\lambda}_{2} \right)\\
\lambda_{2}(t) &=- \overline{\lambda}_{2} (t-\overline{t}) \left(t-\overline{t} - 2\overline{x}_{2}\overline{\lambda}_{2} \right)
\end{aligned}$$

Since we must have $x_{1}(t) \leq  \overline{a}$, the expression of $x_{1}$ yields $t-\overline{t} - 3\overline{x}_{2}\overline{\lambda}_{2} \geq 0$, or $\overline{x}_{2}\overline{\lambda}_{2} \leq \inf_{t_ 2< t <\overline{t}} \frac{t-\overline{t} }{3} \leq  0$ and, thus, $\overline{x}_{2}\overline{\lambda}_{2} = -1$, i.e. $\overline{x}_{2} = -\overline{\lambda}_{2}$, and $\overline{t}-3 \leq t \leq \overline{t}$.

Thus the previous expressions of $x_{1}(t), x_{2}(t), \lambda_{2}(t)$ read:
$$\begin{aligned}
x_{2}(t) &=- \overline{\lambda}_{2} (t-\overline{t} + 1)\\
x_{1}(t) &= \overline{a} - \frac{1}{3} (t-\overline{t})^{2} \left( t-\overline{t} + 3 \right)\\
\lambda_{2}(t) &= -\overline{\lambda}_{2} (t-\overline{t}) \left(t-\overline{t} + 2 \right)
\end{aligned}$$
Elimination of $t-\overline{t}$ in the first equation yields $t-\overline{t} = -1 - \frac{x_{2}}{\overline{\lambda}_{2}}$, and
$$x_{1}(t) = \overline{a} - \frac{1}{3} \left( 1 + \frac{x_{2}}{\overline{\lambda}_{2}}\right)^{2} \left( 2 -\frac{x_{2}}{\overline{\lambda}_{2}} \right).$$

Thus, if $\overline{\lambda}_{2} = +1$ (which implies $\overline{u} = -1$), we get
\begin{equation}\label{branch3:ex2}
x_{1}(t) = \overline{a} - \frac{1}{3} \left( 1 + x_{2} \right)^{2} \left( 2 - x_{2} \right)
\end{equation}
and if $\overline{\lambda}_{2} = -1$ (which implies $\overline{u} = +1$), 
\begin{equation}\label{branch4:ex2}
x_{1}(t) = \overline{a} - \frac{1}{3} \left( 1 - x_{2} \right)^{2} \left( 2 + x_{2} \right).
\end{equation}

\begin{figure}[h]
$$\includegraphics[width=0.6\columnwidth]{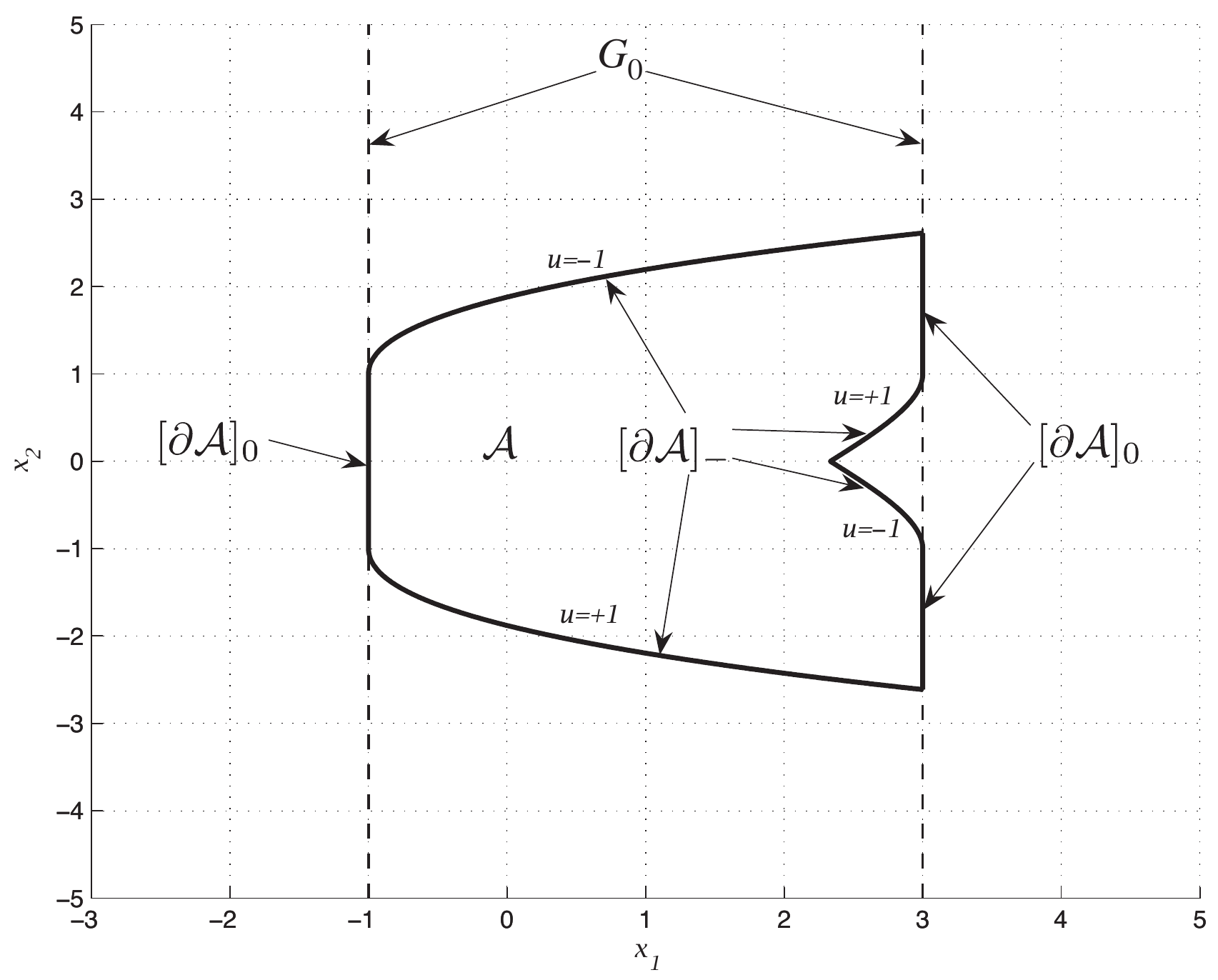}$$
\caption{Admissible set $\AAA$ and barrier for system~\eqref{sys:ex}--\eqref{statecons:ex} with  $\underline{a}=-1$ and $\overline{a}=3$ \label{barr-sing-fig}}
\end{figure}

Remark that the two arcs (\ref{branch3:ex2}) and (\ref{branch4:ex2}) cross at the point $(x_{1},x_{2})=(\overline{a}-\frac{2}{3},0)$. At this point, if $\underline{a} \leq \overline{a}-\frac{2}{3}$, two controls $\overline{u}= + 1$ or $\overline{u}= - 1$ are allowed in order to remain in the admissible set. This means that the barrier (made of the pair of corresponding arcs of integral curves from $(\overline{a}-\frac{2}{3},0)$ to $(\overline{a}, \pm1)$) is not differentiable at the point $(\overline{a}-\frac{2}{3},0)$ where the normals are orthogonal, $(\lambda_{1},\lambda_{2})=(1,\pm 1)$. 

Finally, it is readily seen that the admissible set $\AAA$ is the union of the four subsets:
$$
\begin{aligned}
\AAA =& \{ (x_{1},x_{2})\in \RR^{2}: \underline{a} + \frac{1}{3} \left( 1 + x_{2} \right)^{2} \left( 2 - x_{2} \right)  \leq x_{1} \leq \overline{a} ,  x_{2} \leq -1 \} 
\\
& \cup  
 \{ (x_{1},x_{2})\in \RR^{2}: \underline{a} \leq x_{1} \leq \overline{a} - \frac{1}{3} \left( 1 + x_{2} \right)^{2} \left( 2 - x_{2} \right) , -1 \leq x_{2}  \leq 0 \} 
 \\
 & \cup
 \{ (x_{1},x_{2})\in \RR^{2}: \underline{a} \leq x_{1} \leq \overline{a} - \frac{1}{3} \left( 1 - x_{2} \right)^{2} \left( 2 + x_{2} \right) , 0 \leq  x_{2} \leq 1 \} 
 \\
 & \cup
 \{ (x_{1},x_{2})\in \RR^{2}: \underline{a} + \frac{1}{3} \left( 1 - x_{2} \right)^{2} \left( 2 + x_{2} \right)  \leq x_{1} \leq \overline{a} ,  x_{2} \geq 1 \} 
 \end{aligned}
$$
as depicted in Figure~\ref{barr-sing-fig} for $\underline{a}=-1$ and $\overline{a}=3$.

Note that the admissible set $\AAA$ is not even a connected set when $\underline{a} \geq \overline{a}-\frac{2}{3}$.

\section{Conclusions}\label{sec:Conclus}

This paper has addressed the problem of state and input constrained control for nonlinear
systems with multidimensional constraints. In particular, the admissible region of the
state space, where the state and input constraints can be satisfied for all times, was studied.
A local description of the boundary of the admissible region of the state space was obtained.
This boundary is made of two disjoint parts: the subset of the state constraint boundary on 
which there are trajectories pointing towards the interior of the admissible set or tangentially 
to it; and a barrier, namely a semipermeable surface which is constructed via a minimum-like 
principle. A number of examples was provided to illustrate the results of the paper. A complete characterisation of the admissible region was obtained for these two-dimensional examples. While the theoretical necessary conditions derived in this paper hold for any dimension, higher dimensional systems may exhibit more complex geometric features and such examples will be the subject of future study.

\paragraph{Acknowledgement.}The authors wish to express their warm gratitude to Prof.\ Emmanuel Tr\'{e}lat for fruitful discussions.

\appendix

\section{Compactness of solutions}\label{Append-A}

We prove here the inequalities and compactness results used in Sections \ref{sec:AdmSetTopol} and \ref{sec:BoundAdmSet}. Many related results exist in the literature; see e.g. \cite[Theorem 5.2.1]{Trelat} for a result based on purely functional analytic arguments, and \cite[Chap. 9]{Cesari}, \cite[Section 4.2]{Lee_Markus}, \cite[Section 10.3]{Agrachev} for related results specifically oriented to the study of attainable sets and the existence of optimal controls, based on Filippov's theorem \cite{Filippov_siam}. Those results are scattered in several sources and embedded in slightly different contexts than the one of concern here, and are not always easily identifiable. Thus they are provided here for the sake of completeness and unification. 

We recall the following classical lemma:
\begin{lem}\label{bound-lem}
If assumptions (A1) and (A2) of Section~\ref{sec:ConsDynCon} hold true, equation (\ref{eq:state_space}) admits a unique absolutely continuous integral curve over $[t_0, +\infty)$ for every $u\in \UU$ and every bounded initial condition $x_{0}$, which remains bounded for all finite $t\geq t_0$,
\begin{equation}\label{bound}
\Vert x(t) \Vert \leq \left( (1+ \Vert x_{0} \Vert^{\alpha} )e^{\alpha C (t-t_{0})} -1 \right)^{\frac{1}{\alpha}} \triangleq K(\alpha, t)
\end{equation}
with $\alpha=1$ for condition (A2.i) and $\alpha=2$ for condition (A2.ii).

Moreover, we have
\begin{equation}\label{equicont}
\Vert x(t)-x(s) \Vert \leq C(\alpha)\vert t-s \vert
\end{equation}
for all $t, s \in [t_0,T]$ and all $T>t_0$, where
\begin{equation}\label{Calphabound}
C(\alpha) \triangleq
\sup_{\Vert x \Vert  \leq  K(\alpha, T), u\in U_{1}}  \Vert f(x,u)\Vert < + \infty
\end{equation}
with $\alpha = 1$ (resp. $\alpha= 2$) if condition (A2.i) (resp. (A2.ii)) holds.
\end{lem}
\begin{proof}

In the case (A2.i), using the integral representation of (\ref{eq:state_space}), we get
$$
\begin{aligned}
\Vert x(t)\Vert - \Vert x_0\Vert &\leq \int_{t_{0}}^{t} \Vert f(x(\tau),u(\tau))\Vert d\tau \\
&\leq \int_{t_{0}}^{t}  C(1+ \Vert x(\tau)\Vert) d\tau  \quad \mbox{\textrm{according to (A2.i)}}.
\end{aligned}
$$
Thus,
$$ (1+\Vert x(t)\Vert) \leq (1 + \Vert x_0\Vert) + C \int_{t_{0}}^{t}  (1+ \Vert x(\tau)\Vert) d\tau$$
and, by Gr\"{o}nwall's Lemma \cite{Gron},
$$(1+\Vert x(t)\Vert) \leq (1 + \Vert x_0\Vert) e^{C(t-t_{0})}$$
which readily yields (\ref{bound}) with $\alpha =1$.

In the case (A2.ii), multiplying both sides of (\ref{eq:state_space}) by $x(t)^T$, we get
$$\frac{1}{2} \frac{d}{dt} \left(\Vert x(t)\Vert^2 \right) = x(t)^T f(x(t),u(t))
$$
or, in integral representation, and taking absolute values,
$$\begin{aligned}
\left | \Vert x(t)\Vert^2 - \Vert x_0\Vert^2 \right | &= 2 \left| \int_{t_{0}}^{t}  x(\tau)^T f(x(\tau),u(\tau))d\tau \right| \\
&\leq 2C \int_{t_{0}}^{t} (1+ \Vert x(\tau)\Vert^2) d\tau \quad \mbox{\textrm{according to (A2.ii)}}
\end{aligned}
$$
or
$$ \left(1+\Vert x(t)\Vert^2\right) \leq \left(1 + \Vert x_0\Vert^2\right) +2 C \int_{t_{0}}^{t}  \left(1+ \Vert x(\tau)\Vert^2\right) d\tau.
$$
As before, by Gr\"{o}nwall's Lemma, we get
$$\left(1+\Vert x(t)\Vert^2\right) \leq \left(1 + \Vert x_0\Vert^2 \right) e^{2C(t-t_{0})}$$
which readily yields (\ref{bound}) with $\alpha =2$.

To prove Inequality (\ref{equicont}), let us recall that, for every pair $t, s \in [t_0,T]$ and all $T>t_0$,
$x(t)-x(s) = \int_{s}^{t} f(x(\tau),u(\tau))d\tau$. The continuity of $f$ implies that $C(\alpha) < +\infty$ with $C(\alpha)$ defined by (\ref{Calphabound}). We immediately deduce (\ref{equicont}).
\end{proof}

In the following results we will, without loss of generality
(due to time-invariance), replace the initial time $t_0$ by $0$.
\begin{cor}\label{relatcomp-lem}
Let us denote by $\XX(x_{0})$ the set of integral curves issued from an arbitrary $x_{0}$, $\Vert x_{0}\Vert < \infty$, and satisfying (\ref{eq:state_space}), (\ref{eq:initial_condition}), (\ref{eq:input_constraint}).

If assumptions (A1) and (A2) of Section~\ref{sec:ConsDynCon} hold true, $\XX(x_{0})$ is a subset of  $C^{0}([0,\infty), \RR^{n})$, the space of continuous functions from $[0,\infty)$ to $\RR^{n}$, and is relatively compact with respect to the topology of uniform convergence on $C^{0}([0,T], \RR^{n})$ for all finite $T\geq 0$. In other words, from any sequence of integral curves in $\XX(x_{0})$, one can extract a subsequence whose convergence is uniform on every interval $[0,T]$, with $T\geq 0$ and finite, and whose limit belongs to $C^{0}([0,\infty), \RR^{n})$.
\end{cor}
\begin{proof}
In Lemma~\ref{bound-lem}, inequality (\ref{bound}) means that the restriction of the integral curves of $\XX(x_{0})$ to any finite interval $[0,T]$ is equibounded, and (\ref{equicont}) shows that the same restriction to any finite interval $[0,T]$ of the integral curves of $\XX(x_{0})$ is an equicontinuous set with respect to the topology of uniform convergence on $C^{0}([0,T], \RR^{n})$, for all $T\geq 0$.
The relative compactness results from Ascoli-Arzel\`{a}'s theorem (see e.g. \cite[Chap. III, \S 3, p. 85]{yosida}).
\end{proof}

We now prove the following:
\begin{lem}\label{compact-lem}
Assume that (A1), (A2) and (A3) of Section~\ref{sec:ConsDynCon} hold. Given a compact set $\XX_{0}$ of $\RR^{n}$, the set $\XX\triangleq \bigcup_{x_{0}\in \XX_{0}}\XX(x_{0})$ is compact with respect to the topology of uniform convergence on $C^{0}([0,T], \RR^{n})$ for all $T\geq 0$, namely from every sequence $\{x^{(u_{k},x_{k})}\}_{k\in \NN} \subset \XX$ one can extract a  uniformly convergent subsequence on every finite interval $[0,T]$, whose limit $\xi$ is an absolutely continuous integral curve on $[0,\infty)$, belonging to $\XX$. In other words, there exists $\bar{x}\in \XX_{0}$ and $\bar{u}\in \UU$ such that $\xi(t)= x^{(\bar{u},\bar{x})}(t)$ for almost all $t \geq 0$.
\end{lem}

\begin{proof}
Since $\XX_{0}$ is compact, it is immediate to extend inequalities (\ref{bound}) and (\ref{equicont}) to integral curves with arbitrary $x_{0}\in \XX_{0}$ by taking, in the right-hand side of (\ref{bound}), the supremum over all $x_{0}\in \XX_{0}$. Thus,
by the same argument as in the proof of Corollary~\ref{relatcomp-lem}, using Ascoli-Arzel\`{a}'s theorem, we conclude that $\XX$ is relatively compact with respect to the topology of uniform convergence on $C^{0}([0,T], \RR^{n})$, for all $T\geq 0$. It remains to prove that from every sequence $\{x^{(u_{k},x_{k})}\}_{k\in \NN} \subset \XX$ one can extract a  uniformly convergent subsequence on every finite interval $[0,T]$ whose limit $\xi$ belongs to $\XX$.

To this aim, we first remark that, from the fact that $\XX$ is
relatively compact we have that the limit, $\xi$,
is a continuous function on $[0,T]$ for all $T\geq 0$, and that, for every finite $T$ and every $t\in [0,T]$,
\begin{equation}\label{limX-eq}
\xi(t)= \lim_{k\rightarrow \infty} x_{k} + \lim_{k\rightarrow \infty} \int_{0}^{t} f(x^{(u_{k},x_{k})}(s),u_{k}(s))ds = \bar{x} + \lim_{k\rightarrow \infty} \int_{0}^{t} F_{k}(s)ds
\end{equation}
where the limit is taken over a subsequence and with the notations $\bar{x}=\lim_{k\rightarrow \infty} x_{k}$ and $F_{k}(t)= f(x^{(u_{k},x_{k})}(t),u_{k}(t))$.

We denote by $<v,w>=\sum_{i=1}^{n} v_{i}w_{i}$ the scalar product of the vectors $v$ and $w$ in $\RR^{n}$. Since, for every $k$, the integral curve $x^{(u_{k},x_{k})}$ satisfies
$\dot{x}^{(u_{k},x_{k})}(t)=F_{k}(t)$ for almost every $t$, taking the scalar product of both sides by a function $\varphi \in C^{\infty}([0,\infty), \RR^{n})$ and integrating from $0$ to $T$, yields
$$\int_{0}^{T} <\varphi(t),\dot{x}^{(u_{k},x_{k})}(t)>dt= \int_{0}^{T} <\varphi(t),F_{k}(t)>dt
$$
or, after integration by parts:
$$\begin{aligned}
- \int_{0}^{T} <\dot{\varphi}(t),x^{(u_{k},x_{k})}(t)>dt &+ <\varphi(T),x^{(u_{k},x_{k})}(T)> - <\varphi(0),x^{(u_{k},x_{k})}(0)> \\
&= \int_{0}^{T} <\varphi(t),F_{k}(t)>dt
\end{aligned}
$$
Taking the limits of both sides, according to the uniform boundedness of the integrals, we get, with $\dot{\xi}$ defined as a distribution on $[0,T]$:
\begin{equation}\label{distrib-eq}
\begin{aligned}
\int_{0}^{T} <\varphi(t), \dot{\xi}(t)>dt &\triangleq
- \int_{0}^{T} <\dot{\varphi}(t),\xi(t)>dt + <\varphi(T),\xi(T)> - <\varphi(0),\bar{x}> \\
&= \lim_{k\rightarrow \infty} \int_{0}^{T} <\varphi(t),F_{k}(t)>dt.
\end{aligned}
\end{equation}
In other words, $\dot{\xi}=  \lim_{k\rightarrow \infty}F_{k}$ in the sense of distributions. Moreover, for every $T>0$, restricting $\varphi$ to  $C_{K}^{\infty}([0,T], \RR^{n})$ (the set of infinitely differentiable functions from $[0,T]$ to $\RR^{n}$ with compact support, which is indeed contained in $C^{\infty}([0,\infty), \RR^{n})$)  and using the density of $C_{K}^{\infty}([0,T],\RR^{n})$ in $L^{2}([0,T],\RR^{n})$ (see e.g. \cite{schwartz}), equation (\ref{distrib-eq}) also implies that the sequence $F_{k}$ is weakly convergent in $L^{2}([0,T],\RR^{n})$. Let us denote by $\bar{F}_{T}$ its weak limit in $L^{2}([0,T],\RR^{n})$. We have therefore constructed a collection $\{ \bar{F}_{T}\}_{T>0}$ of weak limits, which uniquely defines a function $\bar{F}$ almost everywhere on the whole interval $[0,\infty)$, whose restriction to any interval $[0,T]$ coincides a.e.\ with $\bar{F}_{T}$, i.e.\ $\bar{F}_{\big| [0,T]} = \bar{F}_{T}$ a.e.. Indeed, taking any pair of intervals $[0,T_1]$ and $[0,T_2]$ with $T_1 \leq T_2$, by the uniqueness of the limits $\bar{F}_{T_{1}}$ and $\bar{F}_{T_{2}}$, the restriction of $\bar{F}_{T_{2}}$ to the interval $[0,T_1]$ coincides almost everywhere with $\bar{F}_{T_{1}}$ and it is readily seen that non uniqueness of $\bar{F}$ would contradict the uniqueness of every $\bar{F}_{T}$.

By Mazur's Theorem (see e.g. \cite[Chapter V, \S1, Theorem 2, p. 120]{yosida}), for every $k$, there exists a sequence $\{\alpha_1^{k}, \ldots, \alpha_{k}^{k}\}$ of non negative real numbers, with $\sum_{i=1}^{k} \alpha_{i}^{k} = 1$, such that the sequence $\tilde{F}_{k}= \sum_{i=1}^{k} \alpha_{i}^{k}F_{i}$ is strongly convergent to $\bar{F}$ in every $L^{2}([0,T],\RR^{n})$ for all finite $T$.
Note that this property \emph{a fortiori} holds true if we replace the sequence $F_{i}$ by any subsequence $F_{i_{j}}$ constructed by selecting a subsequence of indices $i_{j}$ such that, given $\ee >0$, $\sup_{t\in [0,T]} \Vert f(x^{(u_{i_{j}},x_{i_{j}})}(t), u_{i_{j}}(t)) - f(\xi(t), u_{i_{j}}(t))\Vert  < \ee 2^{-j}$ for each $j$, which is indeed possible thanks to the uniform convergence of $x^{(u_{k},x_{k})}$ to $\xi$ and the continuity of $f$.  Note also that the limit $\bar{F}$ remains the same (for convenience of notation, we keep the same symbols for the $\alpha_{j}^{k}$'s, but we remark that these coefficients have to be adapted relative to the new subsequence).

By Minkowski's inequality, we have:
\begin{equation}\label{minko-ineq}
\begin{array}{l}
\ds \left( \int_{0}^{T} \Vert \sum_{j=1}^{k} \alpha_{j}^{k} f(\xi(t),u_{i_{j}}(t)) - \bar{F}(t) \Vert^{2} dt \right)^{\frac{1}{2}}
\\
\ds \leq  \left( \int_{0}^{T} \Vert \sum_{j=1}^{k} \alpha_{j}^{k}\left( f(\xi(t),u_{i_{j}}(t)) - f(x^{(u_{i_{j}},x_{i_{j}})}(t),u_{i_{j}}(t)) \right) \Vert^{2} dt \right)^{\frac{1}{2}}\\
\ds \hspace{2cm}+  \left( \int_{0}^{T} \Vert \sum_{j=1}^{k} \alpha_{j}^{k}f(x^{(u_{i_{j}},x_{i_{j}})}(t),u_{i_{j}}(t)) -\bar{F}(t) \Vert^{2} dt \right)^{\frac{1}{2}}.
\end{array}
\end{equation}
We now prove that the limits of the two terms on the right-hand side of expression~\eqref{minko-ineq}, as $k$ tends to infinity, exist and are equal to 0. The convergence of the second limit to 0 is clearly an immediate consequence of the strong convergence of $\sum_{j=1}^{k} \alpha_{j}^{k}F_{i_{j}}$ to $\bar{F}$.

For the first term on the right-hand side of~\eqref{minko-ineq},
according to the construction of the above subsequence, and using the fact that $\alpha_{j}^{k} \in [0,1]$ for every $j$, we have
\begin{equation}\label{series-ineq}
\begin{array}{l}
\Vert \sum_{j=1}^{k} \alpha_{j}^{k}\left( f(\xi(t),u_{i_{j}}(t)) - f(x^{(u_{i_{j}},x_{i_{j}})}(t),u_{i_{j}}(t)) \right) \Vert \\
\ds \hspace{4cm}  \leq
\sum_{j=1}^{k} \alpha_{j}^{k} \Vert  f(\xi(t),u_{i_{j}}(t)) - f(x^{(u_{i_{j}},x_{i_{j}})}(t),u_{i_{j}}(t)) \Vert \\
\ds \hspace{4cm}  <
\ee \sum_{j=1}^{k} \alpha_{j}^{k} 2^{-j} \\
\ds \hspace{4cm}\leq \ee \left(  \max_{1\leq j\leq k} \alpha_{j}^{k} \right) \sum_{j=1}^{k}  2^{-j} \leq  \ee (1-2^{-k}).
\end{array}
\end{equation}
Thus
\begin{equation}\label{series-comb-ineq}
\begin{array}{l}
\ds \left( \int_{0}^{T} \Vert \sum_{j=1}^{k} \alpha_{j}^{k}\left( f(\xi(t),u_{i_{j}}(t)) - f(x^{(u_{i_{j}},x_{i_{j}})}(t),u_{i_{j}}(t)) \right) \Vert^{2} dt \right)^{\frac{1}{2}} \\
\ds \hspace{4cm} < \left( T \left(  \ee (1-2^{-k})\right)^2 \right)^{\frac{1}{2}} < \ee  \sqrt{T},
\end{array}
\end{equation}
hence, since $\ee$ can be chosen arbitrarily small, the left-hand term in (\ref{series-comb-ineq}) converges to 0 as $k$ tends to infinity.

Therefore, the same holds for the left-hand side of (\ref{minko-ineq}), which proves that $\bar{F} (t)$ belongs  almost everywhere to the closed convex hull of $\{f(\xi(t),u_{i_{j}}(t))\}_{j\in \NN}$ which is contained in $f(\xi(t), U_{1})$ according to (A3).

Finally, again according to (A3), for every $k$, there exists, by the measurable selection theorem \cite{Cast_Val}, $v_{k}\in \UU$ such that 
$$f(\xi(t),v_{k}(t)) = \sum_{j=1}^{k}\alpha_{j}^{k} f(\xi(t),u_{i_{j}}(t))$$ 
for almost all $t$ and,
since strong $L^{2}$ convergence implies pointwise convergence almost everywhere of a subsequence (see e.g. \cite{Kolmo}), we have, from the convergence to 0 of the left-hand side of (\ref{minko-ineq}), that, taking the limit over such a subsequence,
$$\lim_{k\rightarrow \infty}f(\xi(t),v_{k}(t)) = \lim_{k\rightarrow \infty}\sum_{j=1}^{k}\alpha_{j}^{k} f(\xi(t),u_{i_{j}}(t)) = \bar{F}(t)\quad \mbox{a.e.}~t\in [0,\infty).$$
Thus, taking the continuity of $f$ with respect to $u$ into account (see (A1)), we conclude that the sequence $v_{k}$ pointwise converges to some $\bar{u}\in \UU$ such that
$\bar{F}(t)= f(\xi(t),\bar{u}(t))$ for almost all $t$, and therefore that $\xi$ satisfies $\dot{\xi}=f(\xi,\bar{u})$ almost everywhere, with $\xi(0)=\bar{x}\in \XX_0$. By the uniqueness of integral curves of (\ref{eq:state_space}), we conclude that $\xi(t)=x^{(\bar{u},\bar{x})}(t)$ almost everywhere and, thus, that $\xi\in \XX$, which achieves to prove the lemma.
\end{proof}

\section{Needle perturbations and the maximum principle (\cite{PBGM,Lee_Markus,Gam})}\label{Append-B}

\subsection{Perturbations}\label{perturb-subsec-append-b}
Given $\bar{u} \in \UU$ and an integral curve $x^{(\bar{u},\bar{x})}$, we consider a non negative real number $\ee \in [0,\ee_{0}]$ with bounded $\ee_{0}$, an initial state perturbation $h\in \RR^n$ satisfying $\Vert h\Vert \leq H$ and a variation $u_{\kappa,\ee}$ of $\bar{u}$, parameterized by the vector $\kappa \triangleq (v,\tau,l) \in U_{1} \times [0,T] \times [0,L]$ with bounded $T,L$, of the form
\begin{equation}\label{u-var-eq}
u_{\kappa,\ee} \triangleq 
\bar{u} \Join_{(\tau-l\ee)} v \Join_{\tau} \bar{u}
 =
 \left\{ \begin{array}{lcl}
v&\mbox{\textrm{on}}& [\tau-l\ee, \tau[\\
\bar{u}&\mbox{\textrm{elsewhere on}}&[0,T]
\end{array}\right. 
\end{equation}
where $v$ stands for the constant control equal to $v \in U_{1}$ for all $t\in [\tau-l\ee, \tau[$.
We also consider the corresponding integral curve $x^{(u_{\kappa,\ee}, \bar{x}+\ee h)}$ starting from $\bar{x}+\ee h$ and generated by $u_{\kappa,\ee}$.

We indeed have $x^{(u_{\kappa,\ee}, \bar{x}+\ee h)}(t)= x^{(\bar{u},\bar{x}+\ee h)}(t)$ for all $t\in [0, \tau-l\ee [$ and, denoting by $z_{\ee} (\tau-l\ee) \triangleq  x^{(\bar{u},\bar{x}+\ee h)}(\tau-l\ee)$ and 
$z_{\ee}(\tau) \triangleq x^{(u_{\kappa,\ee}, \bar{x}+\ee h)}(\tau)$, we have 
$$x^{(u_{\kappa,\ee}, \bar{x}+\ee h)} = x^{(\bar{u},\bar{x}+\ee h)} \Join_{(\tau-l\ee)} x^{(v, z_{\ee} (\tau-l\ee), \tau - l\ee)} \Join_{\tau} x^{(\bar{u}, z_{\ee}(\tau),\tau)}$$

We also consider the fundamental matrix of the variational equation:
\begin{equation}\label{group}
\frac{d}{dt}\Phi^{\bar{u}}(t,s) = \left( \frac{\partial f}{\partial x}(x^{(\bar{u},\bar{x})}(t), \bar{u}(t))\right) \Phi^{\bar{u}}(t,s), \quad \Phi^{\bar{u}}(s,s)=I_{n}
\end{equation}
where $I_{n}$ is the identity matrix of $\RR^{n}$.

We have the following approximation result (see e.g. \cite[Chapter II, \S 13]{PBGM}, \cite[Chapter 4, p. 248]{Lee_Markus}, \cite{Gam}):
\begin{lem}\label{approx-lem}  

The sequence $\left\{ x^{(u_{\kappa,\ee}, \bar{x}+\ee h)} \right\}_{\ee \geq 0}$ is uniformly convergent  to $x^{(\bar{u},\bar{x})}$ on $[0,T]$ as $\ee$ tends to 0, uniformly with respect to $\kappa$ and $h$.

If, moreover,  $\tau$ is a Lebesgue point of $\bar{u}$, we have, for all $t\in [\tau, T]$:
\begin{equation}\label{approx-eq}
x^{(u_{\kappa,\ee}, \bar{x}+\ee h)}(t) - x^{(\bar{u},\bar{x})}(t) = 
\ee w(t,\kappa,h)
+ O(\ee^{2})
\end{equation} 
\end{lem}
where
\begin{equation}\label{needle-eq}
w(t,\kappa,h) \triangleq  \Phi^{\bar{u}}(t,0)h + l \Phi^{\bar{u}}(t,\tau) \left( f(x^{(\bar{u},\bar{x})}(\tau), v) - f(x^{(\bar{u},\bar{x})}(\tau), \bar{u}(\tau)) \right).
\end{equation} 
The perturbed vector $w(t, \kappa,h)$ at time $t$ associated to $\kappa$ and $h$ is tangent at $x^{(\bar{u},\bar{x})}(t)$ to the curve of perturbed states $\ee \mapsto x^{(u_{\kappa,\ee}, \bar{x}+\ee h)}(t)$ at time $t$. 

It is immediate to see that multiplying $l$ and $h$ by any real number $\eta > 0$ yields the perturbation vector $\eta\cdot w(t, \kappa,h)$. Therefore, the set of such perturbation vectors forms a cone in $\RR^{n}$.

\subsection{The perturbation cone}
We now consider an arbitrary sequence of variations of the form (\ref{u-var-eq}), with associated vector $\chi\triangleq \{(v_{i},\tau_{i},l_{i}) : i=1,\ldots,k\}$, i.e.
$$
u_{\chi,\ee} \triangleq 
\bar{u} \Join_{(\tau_{1}-l_{1}\ee)} v_{1} \Join_{\tau_{1}} \bar{u} \Join_{(\tau_{2}-l_{2}\ee)} v_{2} \Join_{\tau_{2}} \bar{u}  \cdots  \Join_{(\tau_{k}-l_{k}\ee)} v_{k} \Join_{\tau_{k}} \bar{u} 
$$
where the perturbation times  $\tau_{1}< \cdots < \tau_{k}$ are Lebesgue points of $\bar{u}$.
Lemma~\ref{approx-lem} still applies and, for every $t \in [\tau_{k}, T]$, the approximation formula (\ref{approx-eq}) reads:
 \begin{equation}\label{approxmult-eq}
x^{(u_{\chi,\ee}, \bar{x}+\ee h)}(t) - x^{(\bar{u},\bar{x})}(t) = \ee w(t, \chi,h)
+ O(\ee^{2})
\end{equation} 
where
\begin{equation}\label{needlemult-eq}
w(t, \chi,h) \triangleq \Phi^{\bar{u}}(t,0)h + \sum_{i=1}^{k} l_{i} \Phi^{\bar{u}}(t,\tau_{i}) \left( f(x^{(\bar{u},\bar{x})}(\tau_{i}), v_{i}) - f(x^{(\bar{u},\bar{x})}(\tau_{i}), \bar{u}(\tau_{i})) \right).
\end{equation}  
The control variation $u_{\chi,\ee}$ associated to the parameters $\chi$ and $\ee$ is called \emph{needle perturbation}. The cone generated by convex combinations of such needle perturbations at time $t$ is denoted by $\KK_{t}$. Note that we have $\Phi^{\bar{u}}(\tau_{2},\tau_{1}) \KK_{\tau_{1}}\subset \KK_{\tau_{2}}$ for all $0 \leq \tau_{1} < \tau_{2} \leq T$.

\subsection{The maximum principle}

We recall here the following classical result 
(see, e.g. \cite[Theorem 3, Chapter 4, p. 254]{Lee_Markus}).

\begin{thm}[Maximum principle]\label{extrem:thm}
Consider the constrained system (\ref{eq:state_space}), (\ref{eq:initial_condition}), (\ref{eq:input_constraint}). Let $\bar{u}\in \UU$ be such that $x^{(\bar{u},x_{0})}(t_{1}) \in \partial X_{t_{1}}(x_{0})$  for some $t_{1}>0$ (where $\partial X_{t_{1}}(x_{0})$ denotes the boundary of the attainable set defined by~\eqref{attain-set}). Then, there exists a non zero absolutely continuous maximal solution $\eta^{\bar{u}}$ to the adjoint equation
\begin{equation}\label{adjoint-eta}
\dot{\eta}^{\bar{u}}(t)= - \left( \frac{\partial f}{\partial x}(x^{(\bar{u},x_{0})}(t),\bar{u}(t))\right)^T\eta^{\bar{u}}(t), 
\end{equation}
such that
\begin{equation}\label{barriercond-eta} \max_{u\in U_{1}} \left\{(\eta^{\bar{u}}(t))^T f(x^{(\bar{u},x_{0})}(t),u)\right\}=(\eta^{\bar{u}}(t))^T f(x^{(\bar{u},x_{0})}(t),\bar{u}(t))= constant
\end{equation}
for almost all $t\in [0,t_{1}]$.
\end{thm}

\bibliographystyle{plain}

\end{document}